\documentclass[11pt]{article} 
\usepackage{cancel}
\usepackage[english]{babel}
\usepackage[utf8x]{inputenc}
\usepackage{amsmath}
\usepackage{amsthm}
\usepackage{graphicx}
\usepackage{color}

\usepackage{amsfonts}
\usepackage{epstopdf}
\usepackage{float}
\usepackage{hyperref}
\usepackage[dvipsnames]{xcolor}
\usepackage{caption}
\usepackage{subcaption}
\usepackage{tikz}
\usepackage{makeidx}
\usepackage[pagewise]{lineno}
\usepackage{amssymb}

\makeindex

\let \eps \varepsilon

\title{Tracking controllability on moving targets for parabolic equations}
\author{Jone Apraiz\thanks{Department of Mathematics, University of the Basque Country UPV/EHU, Barrio Sarriena s/n, 48940, Leioa, Spain. ORCID: 0000-0001-7866-8412 
E-mail: {\tt jone.apraiz@ehu.eus}}\and Jon Asier Bárcena-Petisco\thanks{Department of Mathematics, University of the Basque Country UPV/EHU, Barrio Sarriena s/n, 48940, Leioa, Spain. ORCID: 0000-0002-6583-866X E-mail: {\tt jonasier.barcena@ehu.eus}.} \and Judit Muñoz-Matute\thanks{Department of Mathematics, University of the Basque Country UPV/EHU, Barrio Sarriena s/n, 48940, Leioa, Spain. ORCID: 0000-0002-1875-8982 E-mail: {\tt judit.munoz@ehu.eus}.}}
\begin{document}
\maketitle
\numberwithin{equation}{section}
\newtheorem{theorem}{Theorem}
\numberwithin{theorem}{section}
\newtheorem{proposition}[theorem]{Proposition}
\newtheorem{conjecture}{Conjecture}
\newtheorem{fact}[theorem]{Fact}
\newtheorem{lemma}[theorem]{Lemma}
\newtheorem{step}{Step}
\newtheorem{corollary}[theorem]{Corollary}
\theoremstyle{remark}
\newtheorem{remark}[theorem]{Remark}
\newtheorem{definition}[theorem]{Definition}
\newtheorem{example}[theorem]{Example}
\newtheorem{hypothesis}{Hypothesis}

\noindent
\textbf{Abstract:} In this paper, we study the tracking controllability of a 1D parabolic-type equation. Notably, with controls acting on the boundary, we seek to approximately control the solution of the equation at specific points of the domain. We prove that acting on one boundary point, we control the solution on one target point, whereas acting on two boundary points, we can control the solution on up to two target points. In order to do so, when the target is fixed, we study the controllability by minimizing the corresponding problem with duality results. Afterwards, we study the controllability on
moving points by applying a transformation that takes the problem to a fixed target. Lastly, we also solve some of these control problems numerically and compute approximations of the solutions and the desired targets, which validates our theoretical methodology.

\noindent
\textbf{Key words:}  approximate controllability,  numerical simulation, parabolic equation, tracking controllability

\noindent
\textbf{AMS subject classification:}  35K40, 35Q93, 65K10, 65N06, 93B05, 93B17.

\noindent
\textbf{Abbreviated title:} Tracking controllability for parabolic equations

\noindent
\textbf{Acknowledgements:}  J.A. and J.A.B.P. were supported by the project PID2024-158206NB-100, funded by MICIU/AEI/10.13039/501100011033 and FEDER, UE, the IMUS - María de Maeztu grant CEX2024-001517-M - Apoyo a Unidades de Excelencia María de Maeztu, funded by MICIU/AEI/10.13039/501100011033 and the Project of Consolidación Investigadora CNS2024-154725, also funded by MICIU/AEI/10.13039/501100011033.  J.A.B.P. was also funded by PID2023-146764NB-I00 funded by MICIU/AEI/10.13039/501100011033 and cofunded by the European Union.
J.M.M. was supported by the Consolidated Research Group MATHMODE (IT1866-26) of the UPV/EHU given by the Department of Education of the Basque Government, the Research Project PID2023-146668OA-I00 and the grant RYC2023-045172-I funded by MICIU/AEI/10.13039/501100011033.

\newpage
\section{Introduction}

In this paper, we study the controllability of the following evolution systems in a one-dimensional spatial domain, $[0,L]$, and for the time interval $[0,T]$, for some $L,T>0$: 
\begin{equation}\label{con:heat0gen}
\begin{cases}
y_{t}-a(t,x)\partial_{xx} y+b(t,x)\partial_x y +c(t,x)y
=0 & \mbox{ in } (0,T)\times(0,L),\\
y(\cdot,0)=0 &\mbox{ on } (0,T),\\
y(\cdot,L)=v& \mbox{ on } (0,T),\\
y(0,\cdot)=0 & \mbox{ on }(0,L),
\end{cases}
\end{equation}
and
\begin{equation}\label{con:heat0genbis}
\begin{cases}
y_{t}-a(t,x)\partial_{xx} y+b(t,x)\partial_x y +c(t,x)y
=0 & \mbox{ in } (0,T)\times(0,L),\\
y(\cdot,0)=v_0 &\mbox{ on } (0,T),\\
y(\cdot,L)=v_L& \mbox{ on } (0,T),\\
y(0,\cdot)=0 & \mbox{ on }(0,L),
\end{cases}
\end{equation} 
where
$a$, $b$ and $c$ are real analytic functions in $[0,T]\times[0,L]$ (see Definition \ref{def:analytic} in the Appendix), satisfying:
\begin{equation}\label{eq:coerca}
 a(t,x)\geq a_0>0,   \qquad \forall t\in[0,T], \quad x\in[0,L],
\end{equation}
for the problem to be well-posed. In these evolution systems, the function $y$ represents a quantity that changes over time and space, for example: the temperature, a concentration of a substance, or the price of a financial asset (see \cite{frey2011}). The coefficient $a$ may model the diffusion and thermal conductivity, whereas coefficients $b$ and $c$ are related to advection-convection and reaction-attenuation effects, respectively. The incorporation of space-time dependent coefficients allows for a more accurate representation of non-homogeneous media in the applications that are mentioned in a few lines below. Observe also, that apart from considering zero initial data in the system, we have introduced some control functions on the boundaries: $v$ on the right part of the boundary in \eqref{con:heat0gen}, and $v_0$ and $v_L$ on the two ends of the domain in \eqref{con:heat0genbis}.

The aim of this work is to study the tracking controllability of these systems. Our objective is to control the solution in the interior of the domain at some specific points that may change smoothly with respect to the time variable. 

First, in Section \ref{sec:trncc}, we will 
study the controllability of the solution at $N\in\mathbb{N}$ fixed points
over the whole time interval $[0,T]$, denoting them by $x_i\in (0,L)$ with target states $w_i\in L^2(0,T)$. We seek to analyze whether there is a control $v\in L^2(0,T)$ for system \eqref{con:heat0gen}, or two controls $v_0\in L^2(0,T)$ and $v_L\in L^2(0,T)$ on the two spatial boundaries for system \eqref{con:heat0genbis}, such that:
\[y(t,x_i)=w_i(t),\,\quad \operatorname{a.e.} \,t\in(0,T),\ \ \forall i=1,\ldots,N,\]
where $y$ is the solution to \eqref{con:heat0gen} and \eqref{con:heat0genbis}, respectively. With this purpose in mind, a duality result will be proved (Lemma \ref{lm:interiordualgen}) for the general case $i=1,\ldots, N$.  The results of the tracking controllability study will be presented in Theorem \ref{tm:interncontrgencon}.

Then, in Section \ref{sec:mtrg}, we will allow  the points which we control to depend on the time variable. We are going to prove that in \eqref{con:heat0gen}, setting only one control $v$ on the right part of the boundary, we can approximately control the solution at the moving target point given by the analytic function $h$ to a target state $w\in L^2(0,T)$; that is,
 \[y(t,h(t))=w(t),\,\quad \operatorname{a.e.} \,t\in(0,T).
 \]
 The result obtained for this problem will be shown in Theorem \ref{tm:contrajh}. Similarly, for system \eqref{con:heat0genbis}, we will see in Theorem \ref{tm:contrajhbis} that we can approximately control (with controls $v_0$ and $v_L$) the solution along two time-dependent analytic trajectories $h_1$ and $h_2$ simultaneously to targets $w_1$ and $w_2$, respectively; that is,
 \[y(t,h_i(t))=w_i(t),\,\quad \operatorname{a.e.} \,t\in(0,T),\,\quad\text{for }\,i=1,2. \]

These problems have aroused the interest of the mathematical community in recent years because of their applications in computer sciences, the study  of dynamical systems, or  chemical reactors, among others. An example of this is that nowadays they are fundamental in modeling biological systems, such as nerve physiology and heart rate regulation (see \cite{ottensen2014}, for example), as well as in economic and financial modeling (notably the Black-Scholes framework, as can be seen in \cite{frey2011}). In the field of robotics, these models are essential for the synthesis of high-level decision strategies, specifically when considering the mechanics of continuous media in flexible structures or soft robotics, where dynamics are rigorously governed by parabolic partial differential equations (see \cite{han2020} or \cite{Gao2023}). For this type of system, the design of stable tracking control laws has advanced significantly through the backstepping method for boundary control, a methodology formalized in \cite{Krstic}.

Regarding the controllability of parabolic equations to specific values (constant targets), the first paper to our knowledge is \cite{laroche2000motion}, where the value on the boundary of the heat equation is controlled through power series. Soon after, the papers \cite{lynch2002flatness} and \cite{dunbar2003motion} were published, where the controllability of the boundary in quasilinear heat equations and the non-linear Stefan equation are studied, respectively. More recently, in \cite{barcena2023tracking}, tracking controllability has been studied for the  heat equation in arbitrary dimensions, and an estimate of the cost of approximately controlling the boundary is given. We would like to highlight that, to our knowledge, this is the first work in which instead of controlling the state on the boundary, we control the state in the interior, at more than one point, and at target points that change smoothly with respect to the time variable.

This scarcity of results involving parabolic equations contrasts with the large amount of results involving hyperbolic equations. Regarding the literature, there are many related works involving hyperbolic equations.  Some of the main ones are \cite{li2010exact} and \cite{li2016exact}, where exact boundary controllability for hyperbolic systems is studied; \cite{gu2011exact} and \cite{gu2013exact}, where hyperbolic systems and unsteady flows are studied on tree-like networks;  \cite{leugering2021nodal}, where control on exact beams is studied;
\cite{saracc2022sidewise}, where the controllability of the wave equation is studied; and \cite{wang2023exact}, where exact boundary controllability of nodal profiles for quasi-linear hyperbolic systems and its asymptotic stability are studied.

In addition, there is a research line that studies the reverse problem, with controls supported on a point in the interior, where the state at time $t=T$ is controlled. As for the heat equation, the main paper is \cite{castro2004unique}, where results are obtained through Fourier series representation and the time analyticity of solutions.
Regarding hyperbolic problems, there are substantially more works. Let us mention here some of the main ones:
\cite{avdonin}, where the authors studied pointwise and internal controllability; \cite{fabre}, where pointwise controllability was taken as the limit of internal controllability in one dimension; and \cite{haraux}, in which pointwise controllability was studied for the vibrations of a plate in two dimensions. More recently, the same was done with moving controls in the paper \cite{castro2013exact}, where the linear wave equation is considered with a control supported on a moving point, and sufficient conditions on the trajectory of the control are obtained in order to have the exact controllability property at a target $T$ for the whole state. Comparing to that paper,  we also use duality, but we provide negative examples on controllability with explicit construction of solutions, and we rely on analyticity and change of variables rather than energy estimates.  It is important to emphasize that our work differs from \cite{castro2013exact} both in the nature of the governing equation and the control mechanism. While \cite{castro2013exact} deals with a hyperbolic equation and an internal moving control, we focus on a 1D parabolic equation with boundary controls aiming to track the solution at specific (fixed or moving) target points. Consequently, our methodology relies on a different framework, combining duality arguments with domain transformations adapted for parabolic operators, rather than the wave-equation techniques used in \cite{castro2013exact}.

Another close but different issue is averaged tracking controllability (note that it is different from averaged control, \cite{barcena2021averaged,zuazua2014averaged,lu2016averaged}), where the whole behavior of the state or of the output variable within the time interval of control is of interest (see \cite{danhane2025averaged} for a specific work related to that).

There are other works in which pointwise control has been studied for other equations, such as \cite{ramos2}. There, the authors proved the existence and uniqueness of a Nash equilibrium for the control of linear partial differential equations of parabolic type (for example, the Burgers equation) and developed an algorithm to approximate the control solution. This work was a continuation of a previous one, \cite{ramos1}. Finally, in the paper \cite{zamorano2024tracking}, the tracking controllability of finite-dimensional linear systems was studied.  

The structure of the rest of the paper is as follows: in Section \ref{sec:prel} we recall some basic notation we will use along the article and present some basic definitions; in Section \ref{sec:trncc} we study the controllability of the solution on fixed points in space; in Section \ref{sec:mtrg} we study the controllability on moving trajectories; in Section \ref{sec:num} we present some numerical simulations; and in Section \ref{sec:op} we suggest some open problems and possible extensions of this work.

\section{Preliminaries and notation}\label{sec:prel}

Let us start by presenting the notation that we  use throughout this article:

\begin{itemize}
\item In this paper, for a linear space $V$, we denote its dual space by $V'$. Moreover, $\langle L,v\rangle_{V'\times V}$ denotes the evaluation of $v$ by the functional $L\in V'$. This will appear when integrating by parts. In the case where the function space is a Hilbert space, we will simplify the notation and write $\langle v,w\rangle_{V}$, for $v,w\in V$, as it is self-dual.
\item To simplify the  notation, in some parts of this  article we will use $y_i=y(\cdot,x_i)$, for $i=1,\ldots,N$.
\item In Section \ref{sec:trncc}, given an interval $(a,b)\subset[0,L]$, we will denote by $1_{(a,b)}(x)$ the indicator function that returns 1 if $x\in (a,b)$ and 0 otherwise.

\end{itemize}

To continue, let us present some basic definitions:

\begin{definition}
    Let $L>0$ and $h_1,\ldots,h_N$ analytic functions, 
satisfying $0<h_1(t)<\cdots<h_N(t)<L$ for all $t\in[0,T]$. Then, system \eqref{con:heat0gen} is approximately controllable in the trajectories $(h_1,\ldots,h_N)$ if for all $(w_1,\ldots,w_N)\in (L^2(0,T))^N$ and $\eps>0$, there exists $v\in L^2(0,T)$ such that the solution of \eqref{con:heat0gen} satisfies
\[\left\|\Big(y(t,h_1(t))-w_1(t),\ldots,y(t,h_N(t))-w_N(t)\Big)\right\|_{(L^2(0,T))^N}<\eps.
\]
A similar definition can be considered for system \eqref{con:heat0genbis}, with the existence of $v$ replaced by the existence of $v_0$ and $v_L$.
\end{definition}

\begin{remark}
The approximate controllability problem in trajectories is a definition that makes sense. Indeed, when the boundary data is in $L^2(0,T)$, both systems
\eqref{con:heat0gen} and \eqref{con:heat0genbis} admit a solution by transposition in $L^2((0,T)\times (0,L))$ (see \cite{lions1972non}, and for a more recent control-oriented explanation, see \cite{fernandez2010boundary}). 
In addition, multiplying $y$ by a cut-off function $\chi$ that is $1$ in $(\eps,L-\eps)$ and $0$ in $[0,\eps/2)\cup(L-\eps/2,L]$, for $\eps>0$ small enough, we obtain, with the usual regularity estimates, that $y\in L^2(0,T;H^1(\eps,L-\eps))$. In particular, $y\in L^2(0,T;C^0([\eps,L-\eps]))$ for all $\eps>0$, so $t\mapsto y(t,h_i(t))$ belongs to $L^2(0,T)$ for all $i=1,\ldots,N$.
\end{remark}

\section{Tracking controllability on fixed target points}\label{sec:trncc}
\paragraph{}
Let us first study the control problem when the trajectories are fixed, that is, $h_i(t)=x_i\in(0,L)$. 
For the sake of simplicity, we denote the differential operator related  to the equation we have posed in \eqref{con:heat0gen} and \eqref{con:heat0genbis} is:
\[\mathcal Ly:=-a(t,x)\partial_{xx} y+b(t,x)\partial_x y +c(t,x)y.\]
Also, we consider that its adjoint operator is:
\[\mathcal L^*p:=-a(t,x)\partial_{xx}p
-(b(t,x)+2\partial_xa(t,x))\partial_xp -(\partial_{xx}a(t,x)+\partial_xb(t,x)-c(t,x))p.\]
Note that since the coefficients of $\mathcal L$ are analytic, so are those of $\mathcal L^*$. Finally, when integrating by parts, the following operator is going to appear naturally:
\begin{equation}\label{eq:defMy}
\mathcal My:=-a(t,x)\partial_xy+(\partial_x a(t,x)+b(t,x))y.
\end{equation}

Let us now define the adjoint system of both \eqref{con:heat0gen} and \eqref{con:heat0genbis}:

\begin{definition}
Let $x_1,\ldots,x_N\in (0,L)$
satisfying $x_1<\cdots<x_N$  and $(f_1,\ldots,f_N)\in (L^2(0,T))^N$. Then,
$p_{f_1,\ldots,f_N}$ denotes the solution of:  \begin{equation}\label{eq:dualsyspointcongen}
\begin{cases}
\displaystyle{-p_{t}+\mathcal L^*p=\sum_{i=1}^Nf_i(t)\delta_{x_i}}& \mbox{ in } (0,T)\times(0,L),\\
p(\cdot,0)=p(\cdot,L)=0& \mbox{ on } (0,T),\\
p(T,\cdot)=0 & \mbox{ on } (0,L),
\end{cases}
\end{equation}
where $\delta_{x_i}$ denotes the Dirac delta function supported at a given point $x_i\in (0,L)$.
\end{definition}

 In  order to address the tracking problem, we resort to Lions' duality theory (see \cite{lions1988controlabilite} or \cite{lions1992pointwise}), also known as Hilbert Uniqueness Method (HUM). The approximate controllability of the system ensures that a dual functional, which we will define in the next lemma, is coercive in the corresponding Hilbert space, in our case $(L^2(0,T))^N$. Consequently, the optimal control that minimizes the tracking error is nothing more than the projection of the optimal adjoint state onto the control space, transforming a trajectory search problem into a minimization problem.
If the system is approximately controllable, then the adjoint operator must be injective, which is translated into a Unique Continuation property, as can be seen in the next Lemma \ref{lm:interiordualgen}. On the other hand, the Lax-Milgram Theorem or basic principles of convex optimization guarantee that the minimum of the functional exists and is unique. Finally, from an intuitive point of view, the reason why the $\delta_{x_i}$ appear in the adjoint system is that we aim to control the state at those points, so in the dual system we must be able to observe perturbations on those points, perturbations that can only affect them if a Dirac mass appears there.

To continue with, let us state and prove the following duality result, based on a unique continuation property:

\begin{lemma}[Duality for interior pointwise controllability]\label{lm:interiordualgen}
Let $x_1,\ldots,x_N\in (0,L)$
satisfying $x_1<\cdots<x_N$. Then, 
\eqref{con:heat0gen} 
is approximately controllable on $(0,T)\times\{x_1,\ldots,x_N\}$
if and only if 
\begin{equation}\label{eq:equivdualgen}
(f_1,\ldots,f_N)\in (L^2(0,T))^N \mbox{ and }\partial_x p_{f_1,\ldots,f_N}(\cdot,L)=0 \Longrightarrow (f_1,\ldots,f_N)=(0,\ldots,0).
\end{equation}
In particular, let $\eps>0$ and 
$(w_1,\ldots,w_N)\in (L^2(0,T))^N$. Let us define:
\begin{equation}\label{eq:funcminimgen}
J(f_1,\ldots,f_N)=\frac{1}{2}\int_{0}^T a(t,L)|\partial_x p_{f_1,\ldots,f_N}(t,L)|^2dt
+ \int_{0}^T\sum_{i=1}^N f_i w_i dt +\eps \|(f_1,\ldots,f_N)\|_{(L^2(0,T))^N}.
\end{equation}
If \eqref{eq:equivdualgen} is satisfied, then $J$ has a unique minimizer in $(L^2(0,T))^N$. Moreover, if $(\tilde f_1,\ldots,\tilde f_N)$ is such minimizer, the control 
 \begin{equation}\label{eq:defminconvgen}
 v=\partial_x p_{\tilde f_1,\ldots,\tilde f_N}(\cdot,L),
\end{equation} 
allows us to obtain
\begin{equation}\label{eq:tarintcongen}
\|(y(\cdot,x_1),\ldots,y(\cdot,x_N))
-(w_1,\ldots,w_N)\|_{(L^2(0,T))^N}<\eps,
\end{equation}
for $y$ the solution of \eqref{con:heat0gen}.

Similarly, \eqref{con:heat0genbis} is approximately controllable on $(0,T)\times\{x_1,\ldots,x_N\}$
if and only if 
\begin{equation}\label{eq:dualtwocontrol}
(f_1,\ldots,f_N)\in (L^2(0,T))^N,\ \partial_x p_{f_1,\ldots,f_N}(\cdot,0)=0  \ 
\mbox{ and }\partial_x p_{f_1,\ldots,f_N}(\cdot,L)=0 \Longrightarrow (f_1,\ldots,f_N)=(0,\ldots,0).
\end{equation}

In particular, let $\eps>0$ and  
$(w_1,\ldots,w_N)\in (L^2(0,T))^N$. Let us define:
\begin{equation}\label{eq:funcminimgen2}
\begin{split}
\hat J(f_1,\ldots,f_N)&=\frac{1}{2}\int_{0}^T(a(t,0)|\partial_x p_{f_1,\ldots,f_N}(t,0)|^2+ a(t,L)|\partial_x p_{f_1,\ldots,f_N}(t,L)|^2)dt\\
&+ \int_{0}^T\sum_{i=1}^N f_i w_i dt +\eps \|(f_1,\ldots,f_N)\|_{(L^2(0,T))^N}.
\end{split}
\end{equation}
If \eqref{eq:dualtwocontrol} is satisfied, then $J$ has a unique minimizer in $(L^2(0,T))^N$. Moreover, if $(\tilde f_1,\ldots,\tilde f_N)$ is such minimizer, the controls
\begin{equation*}
v_0=\partial_x p_{\tilde f_1,\ldots,\tilde f_N}(\cdot,0)
\end{equation*}
and
\begin{equation*}
v_L=\partial_x p_{\tilde f_1,\ldots,\tilde f_N}(\cdot,L),
\end{equation*}
allows us to obtain \eqref{eq:tarintcongen}, for $y$ the solution of \eqref{con:heat0genbis}.

\end{lemma}\noindent
Lemma \ref{lm:interiordualgen} can be proved with
duality arguments similar to those in \cite{lions1992pointwise}
and \cite[Proposition II.4]{barcena2023tracking}:

\begin{proof}[Proof of Lemma \ref{lm:interiordualgen}
]
Let us prove the result for \eqref{con:heat0gen}, being the case of \eqref{con:heat0genbis} analogous. 

\textbf{Step 1. Necessity for approximate controllability.}  To begin with, let us obtain the necessity. For that, let us suppose that \eqref{eq:equivdualgen} is not satisfied. Then, there exists $(f_1,\ldots,f_N)\in (L^2(0,T))^N\setminus\{(0,\ldots ,0)\}$ such that 
$\partial_x p_{f_1,\ldots,f_N}(\cdot,L)=0$. After integrating by parts we obtain that (recall that $\mathcal My$ is introduced in \eqref{eq:defMy}):
\begin{equation*}
\begin{split}
0=&\int_0^T\int_0^L(y_t+\mathcal Ly)p_{f_1,\ldots,f_N}\, dx\, dt\\
=&-\langle y(0,\cdot),p_{f_1,\ldots,f_N}(0,\cdot)\rangle_{L^2(0,L)}
+\langle y(T,\cdot),p_{f_1,\ldots,f_N}(T,\cdot)\rangle_{L^2(0,L)}
\\&-\langle y(\cdot,0),a(\cdot,0)\partial_xp_{f_1,\ldots,f_N}(\cdot,0)\rangle_{L^2(0,T)}
+ \langle v,a(\cdot,L)\partial_xp_{f_1,\ldots,f_N}(\cdot,L)\rangle_{L^2(0,T)}
\\&-\langle [\mathcal M y](\cdot,0),p_{f_1,\ldots,f_N}(\cdot,0)\rangle_{L^2(0,T)}
+\langle [\mathcal M y](\cdot,L),p_{f_1,\ldots,f_N}(\cdot,L)\rangle_{L^2(0,T)}
\\&+\int_0^T\left\langle\sum_{i=1}^Nf_i(t)\delta_{x_i},y\right\rangle_{(C^0([0,L])'\times C^0([0,L])}dt
\\
=&-\langle0,p_{f_1,\ldots,f_N}(0,\cdot)\rangle_{L^2(0,L)}
+\langle y(T,\cdot),0\rangle_{L^2(0,L)}
\\&-\langle 0,a(\cdot,0)\partial_xp_{f_1,\ldots,f_N}(\cdot,0)\rangle_{L^2(0,T)}
+ \langle v,0\rangle_{L^2(0,T)}
-\langle \left[\mathcal My\right](\cdot,0),0\rangle_{L^2(0,T)}
+\langle \left[\mathcal My\right](\cdot,L),0\rangle_{L^2(0,T)}
\\&+\sum_{i=1}^N\int_0^Tf_i(t)y(t,x_i)dt
\\=&\sum_{i=1}^N\int_0^Tf_i(t)y(t,x_i)dt.
\end{split}
\end{equation*}
That is, the set $(y(\cdot,x_1),\ldots, y(\cdot,x_N))\in (L^2(0,T))^N$ is orthogonal to $(f_1,\ldots,f_N)$ regardless of the control, so \eqref{con:heat0gen} is not approximately controllable (as the closure of the reachable space is orthogonal to $(f_1,\ldots,f_N)$,  in particular it does not contain $(f_1,\ldots,f_N)$). Thus, the necessity is obtained.
\paragraph{}
\textbf{Step 2: existence and uniqueness of the minimizer of $J$.}
Let us suppose that \eqref{eq:equivdualgen} is satisfied.  

Let us define,
\begin{equation*}
\begin{cases}
\displaystyle{J_1(f_1,\ldots,f_N)=\frac{1}{2}\int_{0}^Ta(t,L)|\partial_x p_{f_1,\ldots,f_N}(t,L)|^2dt},\\
\displaystyle{J_2(f_1,\ldots,f_N)=
 \int_{0}^T\sum_{i=1}^N f_i w_i dt},\\
\displaystyle{J_3(f_1,\ldots,f_N)=\eps \|(f_1,\ldots,f_N)\|_{(L^2(0,T))^N}}.
\end{cases}
\end{equation*}
First, it is clear that $J$ is convex, as it is the sum of three convex functions: $J_1$ is convex because it is a positive-definite quadratic function
(see \eqref{eq:coerca}); $J_2$ is convex because it is
a linear function; and $J_3$ because it is a multiple of a norm. 
Let us show that $J$ is strictly convex.
Let us suppose that for some $(f_1,\ldots,f_N)\in (L^2(0,T))^N\setminus\{(0,\ldots ,0)\},(g_1,\ldots,g_N)\in (L^2(0,T))^N$ and $\theta\in(0,1)$: 
\begin{equation}\label{eq:convJeq_2}
J\Big(\theta(f_1,\ldots,f_N) + (1-\theta)(g_1,\ldots,g_N)\Big)=\theta J(f_1,\ldots,f_N)+
(1-\theta)J(g_1,\ldots,g_N).
\end{equation}
Then,  since all the functionals are convex, we have:
\begin{equation}\label{eq:convJeqespgen}
\begin{split}
J_i\Big(\theta(f_1,\ldots,f_N) &+ (1-\theta)(g_1,\ldots,g_N)\Big)\\&=\theta J_i(f_1,\ldots,f_N)+
(1-\theta)J_i(g_1,\ldots,g_N), \ \ i=1,2,3. 
\end{split}
\end{equation}
In particular, applying \eqref{eq:convJeqespgen} with $i=3$ we get that:
\[\begin{split}
\|\theta(f_1,\ldots,f_N) &+ (1-\theta)(g_1,\ldots,g_N)\|_{(L^2(0,T))^N}\\&=
\theta\|(f_1,\ldots,f_N)\|_{(L^2(0,T))^N}+(1-\theta)\|(g_1,\ldots,g_N)\|_{(L^2(0,T))^N}.
\end{split}
\]
Since $(L^2(0,T))^N$ is a Hilbert space, by squaring both sides, we get that:
\[\begin{split}
\langle \theta(f_1,\ldots,f_N)&, (1-\theta)(g_1,\ldots,g_N)\rangle_{(L^2(0,T))^N}
\\&=\|\theta(f_1,\ldots,f_N)\|_{(L^2(0,T))^N}\|(1-\theta)(g_1,\ldots,g_N)\|_{(L^2(0,T))^N}.
\end{split}
\]
Thus, using the Cauchy-Schwarz inequality, $(f_1,\ldots,f_N)$ and $(g_1,\ldots,g_N)$ are proportional with a positive constant, that is, 
there exists $c\geq0$ such that: 
\[(g_1,\ldots,g_N)=c (f_1,\ldots,f_N).
\]
In addition, if we have \eqref{eq:convJeq_2}, then, using \eqref{eq:convJeqespgen} with $i=1$:
\begin{multline}\label{eq:parpthetafgen}
\frac{1}{2}\int_{0}^Ta(t,L)|\partial_x p_{\theta(f_1,\ldots,f_N)+(1-\theta)(g_1,\ldots,g_N)
}(t,L)|^2dt\\= \frac{\theta}{2}\int_{0}^T a(t,L)|\partial_x p_{f_1,\ldots,f_N
}(t,L)|^2dt
+\frac{1-\theta}{2}\int_{0}^T a(t,L)|\partial_x p_{g_1,\ldots,g_N
}(t,L)|^2dt.
\end{multline}
Since $(f_1,\ldots,f_N)\mapsto \partial_x p_{f_1,\ldots,f_N}$ is linear:
\begin{equation*}
\begin{split}
(\theta+c(1-\theta))^2\int_0^Ta(t,L)|\partial_x p_{f_1,\ldots,f_N}(t,L)|^2dt 
&=\int_0^Ta(t,L)|(\theta+c(1-\theta))\partial_x p_{f_1,\ldots,f_N}(t,L)|^2dt 
\\&=\int_0^Ta(t,L)|\partial_x p_{(\theta+c(1-\theta))(f_1,\ldots,f_N)}(t,L)|^2dt 
\\&=\int_0^Ta(t,L)|\partial_x p_{\theta(f_1,\ldots,f_N)+(1-\theta)(g_1,\ldots,g_N)}(t,L)|^2dt. \\
\end{split}
\end{equation*}
Consequently, using \eqref{eq:parpthetafgen} we obtain that:
\begin{equation}\label{eq:thetapxgen}
\begin{split}
(\theta+c(1-\theta))^2\int_0^T&a(t,L)|\partial_x p_{f_1,\ldots,f_N}(t,L)|^2dt 
\\&=\theta\int_{0}^T a(t,L)|\partial_x p_{f_1,\ldots,f_N
}(t,L)|^2dt
+(1-\theta)\int_{0}^T a(t,L)|\partial_x p_{g_1,\ldots,g_N
}(t,L)|^2dt\\
&=\theta\int_{0}^T a(t,L)|\partial_x p_{f_1,\ldots,f_N
}(t,L)|^2dt
+(1-\theta)\int_{0}^T a(t,L)|\partial_x p_{c(f_1,\ldots,f_N)
}(t,L)|^2dt
\\&=\theta \int_0^Ta(t,L)|\partial_x p_{f_1,\ldots,f_N}(t,L)|^2dt+(1-\theta)c^2
\int_0^Ta(t,L)|\partial_x p_{f_1,\ldots,f_N}(t,L)|^2dt.
\end{split}
\end{equation}
Since $(f_1,\ldots,f_N)\neq (0,\ldots ,0)$, using \eqref{eq:equivdualgen}, we obtain that
$\partial_x p_{\theta(f_1,\ldots,f_N)}\neq0$. Consequently, from \eqref{eq:thetapxgen} we get that: 
\[(1\cdot\theta+c(1-\theta))^2=1^2\cdot \theta + c^2(1-\theta).
\]
Due to the fact that squaring is a strictly convex function, $c=1$ and therefore $(f_1,\ldots,f_N)=(g_1,\ldots,g_N)$,
which shows that $J$ is strictly convex.

Moreover, by using again the Cauchy-Schwarz inequality, we can see that $J$ is coercive.
Indeed:
\[J(f_1,\ldots,f_N)\geq \frac{\eps}{2}\|(f_1,\ldots,f_N)\|_{(L^2(0,T))^N} -\frac{1}{\eps}\|(w_1,\ldots,w_N)\|_{(L^2(0,T))^N}.
\]
Thus, a unique minimizer $(\tilde f_1,\ldots,\tilde f_N)$ exists for $J$. 

\textbf{Step 3. Sufficiency for approximate controllability.}
Let us consider $y$ the solution of \eqref{con:heat0gen} with control 
$v=\partial_x p_{\tilde f_1,\ldots,\tilde f_N}$. 
If  $(f_1,\ldots,f_N)\in (L^2(0,T))^N$, we have that:
\begin{equation*}
\begin{split}
0=&\int_0^T\int_0^L(y_t+\mathcal Ly)p_{f_1,\ldots,f_N}\,dx\, dt\\
=&-\langle0,p_{f_1,\ldots,f_N}(0,\cdot)\rangle_{L^2(0,L)}
+\langle y(T,\cdot),0\rangle_{L^2(0,L)}
\\&-\langle 0,a(\cdot,0)\partial_xp_{f_1,\ldots,f_N}(\cdot,0)\rangle_{L^2(0,T)}
+ \langle \partial_x p_{\tilde f_1,\ldots,\tilde f_N},a(\cdot,L)\partial_x p_{f_1,\ldots,f_N}
\rangle_{L^2(0,T)}
\\&-\langle [\mathcal My](\cdot,0),0\rangle_{L^2(0,T)}
+\langle [\mathcal My](\cdot,L),0\rangle_{L^2(0,T)}
\\&+\sum_{i=1}^N\int_0^T f_i(t)y_i(t)dt;
\\=&\left\langle \partial_x p_{\tilde f_1,\ldots,\tilde f_N},a(\cdot,L)\partial_x p_{f_1,\ldots,f_N}
\right\rangle_{L^2(0,T)}
+\sum_{i=1}^N\int_0^T f_i(t)y_i(t)dt.
\end{split}
\end{equation*}
Consequently, we obtain that: 
\begin{equation*}
\left\langle \partial_x p_{\tilde f_1,\ldots,\tilde f_N},a(\cdot,L)\partial_x p_{f_1,\ldots,f_N}\right\rangle_{L^2(0,T)}
=-\sum_{i=1}^N\int_0^T f_i(t)y_i(t)dt.
\end{equation*}
Therefore, we have that, for $h>0$ a small parameter:
\begin{equation*}
\begin{split}
0\leq &J((\tilde f_1,\ldots,\tilde f_N)\pm h(f_1,\ldots,f_N)
)- J(\tilde f_1,\ldots,\tilde f_N)\\
=&\pm h \left\langle \partial_x p_{\tilde f_1,\ldots,\tilde f_N},a(\cdot,L)\partial_x p_{f_1,\ldots,f_N}\right\rangle_{L^2(0,T)}
\pm h\sum_{i=1}^N\int_0^T f_i(t)w_i(t)dt
\\&+\eps\left(\|(\tilde f_1,\ldots,\tilde f_N)\pm h(f_1,\ldots,f_N)
\|_{(L^2(0,T))^N}
-\|(\tilde f_1,\ldots,\tilde f_N)\|_{(L^2(0,T))^N}\right)+O(h^2)
\\=&\mp h\sum_{i=1}^N\int_0^Tf_i(t)(y_i(t)-w_i(t))dt
\\&+\eps\left(\|(\tilde f_1,\ldots,\tilde f_N)\pm h(f_1,\ldots,f_N)
)\|_{(L^2(0,T))^N}
-\|(\tilde f_1,\ldots,\tilde f_N)\|_{(L^2(0,T))^N}\right)+O(h^2).
\end{split}
\end{equation*}

Thus, using the triangular inequality, we obtain from this estimate that:
\begin{equation*}
\left|\sum_{i=1}^N\int_0^Tf_i(t)(y_i(t)-w_i(t))dt\right|
\leq \eps\|(f_1,\ldots,f_N)\|_{(L^2(0,T))^N}
+O(h),
\end{equation*}
which implies, taking the limit as $h\to 0$, for any 
$(f_1,\ldots,f_N)$:
\begin{equation*}
\left|\sum_{i=1}^N\int_0^Tf_i(t)(y_i(t)-w_i(t))dt\right|
\leq \eps\|(f_1,\ldots,f_N)\|_{(L^2(0,T))^N}.
\end{equation*}
Consequently, using duality, we obtain \eqref{eq:tarintcongen}, proving the sufficiency of \eqref{eq:equivdualgen} for approximate controllability (as we have found a control for any $\eps>0$ and $(w_1,\ldots,w_N)\in L^2(0,T)$). 
\end{proof}

Let us now prove the main result of this section:

\begin{theorem}[Simultaneous interior pointwise control]
\label{tm:interncontrgencon}
Let $x_1,x_2,x_3\in (0,L)$
satisfying $$x_1<x_2<x_3.$$ Then,
\begin{enumerate}
\item\label{it:onetargonecongen}  
System \eqref{con:heat0gen} is  approximately controllable
on $(0,T)\times\{x_1\}$.

\item\label{it:twotargetgen} System \eqref{con:heat0genbis} is  approximately controllable on $(0,T)\times\{x_1,x_2\}$.

\item \label{it:twotargonecongen}  System \eqref{con:heat0gen} is not approximately controllable
on $(0,T)\times\{x_1,x_2\}$.
\item \label{it:threetargetgen} System \eqref{con:heat0genbis} is not 
approximately controllable on $(0,T)\times\{x_1,x_2,x_3\}$.

\end{enumerate}
\end{theorem}

\begin{proof}
The proof of Theorem \ref{tm:interncontrgencon} is based on the  duality provided by Lemma \ref{lm:interiordualgen}.

\begin{itemize}

\item Implication \ref{it:onetargonecongen}
is equivalent by Lemma \ref{lm:interiordualgen}
to proving that if $f_1\in L^2(0,T)$
satisfies $\partial_{x}p_{f_1}(\cdot,L)=0$, then $f_1=0$
in $L^2(0,T)$. For that, we are going to see that $p_{f_1}$ is null at both sides of $x=x_1$. Let us start with $x>x_1$. Let us define  $\bar t=\inf\{t \in [0,T]: p_{f_1}=0\mbox{ on }(t,T)\times(x_1,L)\}$. Let us prove by contradiction that $\bar t=0$. \\Let us suppose that $\bar t>0$.  We define:
$$\tilde p_{f_1}(t,x):=\begin{cases}
p_{f_1}(t,x) & \mbox{for } (t,x) \in (0,T)\times(x_1,L),\\
0 & \mbox{for } (t,x) \in (0,T)\times[L,+\infty).
\end{cases}$$
By continuity, $\tilde p_{f_1}(\bar t,\cdot)=0$. Moreover, let $\eps>0$ be small enough so that there exists the prolongation of $\mathcal L^*$ to $[\bar t-\eps,T]\times[x_1,L+\eps]$ as an analytic operator by extending its coefficients (see Proposition \ref{prop:analytic} and Remark \ref{rk:coerc} in the Appendix). Then, using that 
since the coefficients are analytic, we have that:
\[-(\tilde p_{f_1})_t+\mathcal L^*\tilde p_{f_1}=0 \qquad\mbox{ on }(\bar t-\eps,T)\times(x_1,L+\eps).\]
Hence, by Holmgren's Uniqueness Theorem  (see \cite[Thm. 8.6.5]{hormander1983analysis}), since $\tilde p_{f_1}$ is null on $(\bar t-\eps,T)\times (L,L+\eps)$, $\tilde p_{f_1}$ is null on $(\bar t-\eps,T)\times (x_1,L+\eps)$, arriving at a contradiction with the definition of $\bar t$, so $\bar t=0$. 
\\This implies that, by continuity of the solution, $p_{f_1}(\cdot,0)=0$.
 Moreover, from
$p_{f_1}(\cdot,0)=p_{f_1}(\cdot,x_1)=0$, using the uniqueness result of the (backward) heat equation with Dirichlet boundary conditions,
we obtain that $p_{f_1}=0$ in $(0,T)\times(0,x_1)$.
 Consequently, from $p_{f_1}=0$ in $(0,T)\times (0,x_1)$ and $(0,T)\times(x_1,L)$,
 we obtain that $p_{f_1}=0$ in $(0,T)\times(0,L)$, and therefore $f_1=0$ in $L^2(0,T)$.

\item 
Implication \ref{it:twotargetgen} is equivalent 
by Lemma \ref{lm:interiordualgen} to proving that
if $f_1,f_2\in L^2(0,T)$ 
 satisfy $\partial_{x}p_{f_1,f_2}(\cdot,0)=0$ and 
$\partial_{x}p_{f_1,f_2}(\cdot,L)=0$, then $f_1=f_2=0$. 
As in the previous item, from $\partial_{x}p_{f_1,f_2}(\cdot,0)=0$ we obtain that $p_{f_1,f_2}$
is null in $(0,T)\times (0,x_1)$, and from $\partial_{x}p_{f_1,f_2}(\cdot,L)=0$  
that $p_{f_1,f_2}$ is null in $(0,T)\times(x_2,L)$.
Consequently, if  $\partial_{x}p_{f_1,f_2}(\cdot,0)=0$ and 
$\partial_{x}p_{f_1,f_2}(\cdot,L)=0$,
$p_{f_1,f_2}(\cdot,x_1)=p_{f_1,f_2}(\cdot,x_2)=0$,
so $p_{f_1,f_2}=0$ also in $(0,T)\times(x_1,x_2)$. Thus
$p_{f_1,f_2}=0$ in $(0,T)\times(0,L)$, and in particular $f_1=f_2=0$.

\item Implication \ref{it:twotargonecongen} is equivalent by Lemma \ref{lm:interiordualgen} to proving
that for some $(f_1,f_2)\in (L^2(0,T))^2\setminus\{(0,0)\}$ 
the solution of \eqref{eq:dualsyspointcongen} for $N=2$
satisfies $\partial_{x}p(\cdot,L)=0$.
One of such solutions is given by:

\begin{equation}\label{def:ppartida4}
p(t,x)=\begin{cases}
\tilde p(t,x)& \mbox{ for } (t,x) \in (0,T)\times [0,x_1],\\
\hat p(t,x) & \mbox{ for } (t,x) \in (0,T)\times (x_1,x_2],\\
0 & \mbox{ for } (t,x) \in (0,T)\times (x_2,L],
\end{cases}
\end{equation}

for $\tilde p$ the solution of:
\begin{equation*}
\begin{cases}
-\tilde p_{t}+\mathcal L^*\tilde p=0&
 \mbox{ in } (0,T)\times(0,x_1),\\
\tilde p(\cdot,0)=0& \mbox{ on } (0,T),\\
\tilde p(\cdot,x_1)=1& \mbox{ on } (0,T),\\
\tilde p(T,\cdot)=0 & \mbox{ on } (0,x_1),
\end{cases}
\end{equation*}
and $\hat p$ the solution of:
\begin{equation*}
\begin{cases}
-\hat p_{t}+\mathcal L^*\hat p=0&
 \mbox{ in } (0,T)\times(x_1,x_2),\\
\hat p(\cdot,x_1)=1& \mbox{ on } (0,T),\\
\hat p(\cdot,x_2)=0& \mbox{ on } (0,T),\\
\hat p(T,\cdot)=0 & \mbox{ on } (x_1,x_2).
\end{cases}
\end{equation*}
We will see in the next lines that \eqref{def:ppartida4} satisfies \eqref{eq:dualsyspointcongen} for:
\begin{equation*}
\begin{split}
f_1&=-a(\cdot,x_1)\partial_x\hat p(\cdot,x_1)
+a(\cdot,x_1)\partial_x\tilde p(\cdot,x_1),\\ f_2&=a(\cdot,x_2)\partial_x\hat p(\cdot,x_2).
 \end{split}
\end{equation*}

Since the function $p$ does not have any irregularity in the time variable, we have that:
\[p_t=\tilde p_t 1_{(0,x_1)}+\hat p_t1_{(x_1,x_2)}.\]
Moreover, since $p$ is continuous and differentiable almost everywhere, we have the following derivative as distributions:
\[\partial_x p= \partial_x\tilde p1_{(0,x_1)}+\partial_x\hat p1_{(x_1,x_2)}.
\]
Consequently, we have the following result with respect to the space variable:
\[\partial_{xx}p=\partial_{xx}\tilde p 1_{(0,x_1)}+
(\partial_x\hat p(\cdot,x_1)-\partial_x\tilde p(\cdot,x_1))\delta_{x_1}
+\partial_{xx}\hat p1_{(x_1,x_2)}
-\partial_x\hat p(\cdot,x_2)\delta_{x_2}.\]

Therefore, \eqref{def:ppartida4} satisfies \eqref{eq:dualsyspointcongen} using the chosen functions $f_1$ and $f_2$. It can also be shown that $\partial_{x}p(\cdot,L)=0$ for the  function $p$ defined in \eqref{def:ppartida4}.

To conclude, we have to verify that $(f_1,f_2)\in(L^2(0,T))^2\setminus\{(0,0)\}$. First of all, let us see that $f_1,f_2\in L^2(0,T)$. For that, as $a$ is analytic, it suffices to prove that $\partial_x\tilde p(\cdot,x_1),\partial_x\hat p(\cdot,x_1), \partial_x\hat p(\cdot,x_2)\in L^2(0,T)$. To this end, we introduce the affine function
\[
\ell(x)=\frac{x-x_1}{x_2-x_1},
\]
and define
\[
v:=\hat p - (1-\ell).
\]
Then, since $\hat p(\cdot,x_1)=1$ and $\hat p(\cdot,x_2)=0$, $v$ satisfies homogeneous Dirichlet boundary conditions,
\[
v(\cdot,x_1)=v(\cdot,x_2)=0,
\]
and solves a backward parabolic equation of the form
\[
-v_t+\mathcal L^*v=g
\]
in $(0,T)\times(x_1,x_2)$, where $g=-\mathcal L^*(1-\ell)$ is a smooth
function. Standard parabolic regularity theory (see \cite[Chapter 7]{Evans}) yields
\[
v\in L^2(0,T;H^2(x_1,x_2))
\cap H^1(0,T;L^2(x_1,x_2)).
\]
Hence,
\[
\hat p\in L^2(0,T;H^2(x_1,x_2)).
\]
By the trace theorem,
\[
|\partial_x\hat p(t,x_1)|^2
+
|\partial_x\hat p(t,x_2)|^2
\leq C \|\hat p(t,\cdot)\|_{H^2(x_1,x_2)}^2
\]
for a.e. $t\in(0,T)$. Integrating over $(0,T)$ gives
\[
\partial_x\hat p(\cdot,x_1),
\ \partial_x\hat p(\cdot,x_2)
\in L^2(0,T).
\]

The same argument applied to $\tilde p$ yields
\[
\partial_x\tilde p(\cdot,x_1)\in L^2(0,T).
\]
Therefore,
\[
f_1,f_2\in L^2(0,T).
\]

Finally, $(f_1,f_2)\neq(0,0)$. Indeed, if $f_1=f_2=0$, then the jumps of the
fluxes at $x_1$ and $x_2$ vanish, and thus $p$ would satisfy the homogeneous
adjoint problem associated with \eqref{eq:dualsyspointcongen}. By uniqueness,
this would imply $p=0$, contradicting the fact that $p(\cdot,x_1)=1$.
Hence,
\[
(f_1,f_2)\in (L^2(0,T))^2\setminus\{(0,0)\}.
\]

\item Implication \ref{it:threetargetgen} is equivalent by Lemma \ref{lm:interiordualgen} to proving
that for some\\ $(f_1,f_2,f_3)\in (L^2(0,T))^3\setminus\{(0,0,0)\}$ 
the solution of \eqref{eq:dualsyspointcongen} for $N=3$ satisfies $\partial_{x}p(\cdot,0)=\partial_{x}p(\cdot,L)=0$.
One of such solutions is given by:

\begin{equation}\label{def:ppartida4gen} 
p(t,x)=\begin{cases} 0 & \mbox{ for } (t,x) \in (0,T)\times [0,x_1],\\ \tilde p(t,x) & \mbox{ for } (t,x) \in (0,T)\times (x_1,x_2],\\ \hat p(t,x) & \mbox{ for } (t,x) \in (0,T)\times (x_2,x_3],\\ 0 & \mbox{ for } (t,x) \in (0,T)\times (x_3,L], 
\end{cases} 
\end{equation}

for $\tilde p$ the solution of:
\begin{equation*}
\begin{cases}
-\tilde p_{t}+\mathcal L^*\tilde p=0&
 \mbox{ in } (0,T)\times(x_1,x_2),\\
\tilde p(\cdot,x_1)=0& \mbox{ on } (0,T),\\
\tilde p(\cdot,x_2)=1& \mbox{ on } (0,T),\\
\tilde p(T,\cdot)=0 & \mbox{ on } (x_1,x_2),
\end{cases}
\end{equation*}
and $\hat p$ the solution of:
\begin{equation*}
\begin{cases}
-\hat p_{t}+\mathcal L^*\hat p=0&
 \mbox{ in } (0,T)\times(x_2,x_3),\\
\hat p(\cdot,x_2)=1& \mbox{ on } (0,T),\\
\hat p(\cdot,x_3)=0& \mbox{ on } (0,T),\\
\hat p(T,\cdot)=0 & \mbox{ on } (x_2,x_3).
\end{cases}
\end{equation*}
In fact, using a similar argument to the previous implication, it can be seen that \eqref{def:ppartida4gen} satisfies \eqref{eq:dualsyspointcongen} for:
\begin{equation*}
\begin{split}
f_1&=-a(\cdot,x_1)\partial_x\tilde p(\cdot,x_1),\\ 
f_2&=-a(\cdot,x_2)\partial_x\hat p(\cdot,x_2)
+a(\cdot,x_2)\partial_x\tilde p(\cdot,x_2),\\ f_3&=a(\cdot,x_3)\partial_x\hat p(\cdot,x_3).
 \end{split}
\end{equation*}

\end{itemize}
Moreover, we can easily verify that $\partial_{x}p(\cdot,0)=\partial_{x}p(\cdot,L)=0$ for the function $p$ defined in \eqref{def:ppartida4gen}. Finally, we can verify that $(f_1,f_2,f_3)\in (L^2(0,T))^3\setminus\{(0,0,0)\}$ as in the previous item. 
\end{proof}

\section{Controllability of trajectories on moving target points}\label{sec:mtrg}
\paragraph{}
In this section, we analyze the problem of controlling the system on target points that do not remain fixed. In the previous section we have seen that with one boundary control, as in \eqref{con:heat0gen}, we can at most control one target, whereas with two boundary controls, as in \eqref{con:heat0genbis}, we can at most control two target points. In this section,  we are going to prove that a similar behavior happens when the target points are not fixed.

\subsection{Controlling one target point}
Let us first prove controllability with one target point:

\begin{theorem}\label{tm:contrajh}
Consider the system \eqref{con:heat0gen}. 
Also, consider an analytic function $h$ such that $\min_{[0,T]}h(t)>0$ and $\max_{[0,T]}h(t)<L$.
Then, for all $w\in L^2(0,T)$ and $\eps>0$ there exists a control $v$ such that:
\[\|y(t,h(t))-w(t)\|_{L^2(0,T)}<\eps,\]
where $y$ is the solution of \eqref{con:heat0gen}. 
\end{theorem}
To prove Theorem \ref{tm:contrajh}, we propose a diffeomorphism so that the controlled trace does not change its position in the time variable, and then apply Theorem \ref{tm:interncontrgencon}.

\begin{proof}
First, by a simple change of variable, we can assume from now on that $L=1$. Indeed, this can be done by defining $\tilde x=x/L$, as we get a system of the form \eqref{con:heat0gen} with coefficient $a$ replaced by $L^2a$, $b$ replaced by $Lb$ and $h(t)$ replaced by $h(t)/L$. In particular, this does not affect the analyticity of the target nor that of the coefficients. Thus, we may suppose that the spatial domain is $[0,1]$. Notably, we suppose that $h$ is analytic and that there exist $m>0$ and $M<1$ such that $h(t)\in[m,M]$ for all $t\in[0,T]$. From now on, we consider $\displaystyle{m=\min_{t\in[0,T]}h(t)}$ and $\displaystyle{M=\max_{t\in[0,T]}h(t)}$. 

We propose the following change of variables:
\begin{equation}\label{chi_1}
    \chi(t,x)= \alpha(t)x^n+\beta(t)x,
\end{equation}
for $\alpha,\beta\in C^\infty([0,T])$ analytic functions, 
and $n\in\mathbb N$ to be fixed. 
We are going to find $\alpha$, $\beta$ and $n$ such that:
\begin{itemize}
    \item $\chi(t,0)=0$,
    \item $\chi(t,h(t))=m$,
    \item $\chi(t,1)=1$,
    \item There is $c>0$ such that   $\chi_x(t,x)>c$ 
    for all $t\in[0,T]$ and $x\in[0,L]$. 
\end{itemize}
Indeed, the first and third one ensures that the new domain is $[0,1]$ for all time $t\in[0,T]$, the third one ensures that the new target is a fixed point $x=m$, and the fourth one ensures that for all $t\in[0,T]$ the function $x\mapsto \chi(t,x)$ is a diffeomorphism from $[0,1]$ to $[0,1]$.

From \eqref{chi_1} and the previous second and third conditions, we have to solve the following system for all $t\in[0,T]$:
\begin{equation*}
    \begin{cases}
        \alpha(t)+\beta(t)=1,\\
        \alpha(t)(h(t))^n+\beta(t)h(t)=m.
    \end{cases}
\end{equation*}
The solution is given by:
\[\alpha(t)=\frac{h(t)-m}{h(t)-(h(t))^n}, \ \ \ \beta(t)=\frac{m-(h(t))^n}{h(t)-(h(t))^n}.\]
Note that $\alpha(t)\geq0$. Also, for $n$ large enough, $\beta(t)>0$. Indeed, since $M<1$,
it suffices to pick $n$ so that:
\begin{equation}\label{est:mMalpha}
m-M^n>0.
\end{equation}
As $\alpha(t)\geq0$,
\eqref{est:mMalpha} implies that:
\begin{equation}\label{est:chitx}
\chi_x(t,x)\geq \beta(t)\geq c(m,M,n)=
\min_{s\in[m,M]}\frac{m-s^n}{s-s^n}>0,
\qquad \forall t\in[0,T],\quad \forall x\in [0,1],
\end{equation}
so all the properties are satisfied.

Considering the fourth property, we can invert the change of variables in $x$, which yields the analytic function $\eta(t,x)$ such that $$\chi(t,\eta(t,x))=\eta(t,\chi(t,x))=x.$$ \noindent
Note that the function $\eta(t,x)$ is analytic. Indeed, let us define
\[F(t,x,y)=\chi(t,y)-x.\]
As,
\[|\partial_yF(t,x,y)|=|\partial_y\chi(t,y)|>0,\]
the Implicit Function Theorem ensures that $\eta$ is analytic
(see Theorem 8.6 of Chapter 1 in \cite{kaup1983holomorphic}). 

This allows us to make a change of variables. If $y$ is a solution of
\eqref{con:heat0gen}, then let us obtain the equation that $z(t,x)=y(t,\chi(t,x))$ satisfies. 
As $\eta$ is the spatial inverse of $\chi$,
 we have the following:
\[y(t,x)=z(t,\eta(t,x)).
\]
Thus, we find that $z$ satisfies:
\begin{equation}\label{eq:newzed}
\begin{split}
z_t+\eta_t\partial_xz(t,x)&-a(t,x)\eta_x^2\partial_{xx}z(t,x)-a(t,x)\eta_{xx}\partial_xz(t,x)\\&+b(t,x)\eta_x\partial_{x}z(t,x)+c(t,x)z(t,x)=0. 
\end{split}
\end{equation}

Now, if we define the coefficients 
\[\tilde a(t,x)=a(t,\chi(t,x))\eta_{x}^2, \ \ \tilde b(t,x)=\eta_t-a(t,\chi(t,x))\eta_{xx}+b(t,\chi(t,x))\eta_x \ \text{and} \ \tilde c(t,x)=c(t,\chi(t,x)),\]
we can observe that they are analytic. Thus, \eqref{eq:newzed} can be rewritten as:
\begin{equation}\label{eq:zedtilde}
z_t-\tilde a(t,x)\partial_{xx}z+\tilde b(t,x)\partial_xz+\tilde c(t,x)z=0.
\end{equation}

In order to prove a similar condition to \eqref{eq:coerca} for \eqref{eq:zedtilde}, recalling that
\[\tilde a(t,x)= a(t,\chi(t,x))\cdot\left[\eta_x(t,\chi(t,x))
\right]^2\]
and knowing that $\eta(t,\chi(t,x))=x$,
we have
\[\eta_x(t,\chi(t,x))=\frac{1}{\chi_x(t,x)}.\]
Thus, using \eqref{eq:coerca} and \eqref{est:chitx}, we can ensure that there exists $\tilde a_0>0$ such that:
\[|\tilde a(t,x)|\geq\tilde a_0.\]
Indeed, note that once $h$ is fixed, $m$ and $M$ are fixed, and $n$ is also fixed by choosing it large enough. Consequently, once $h$ is fixed, we obtain $\chi_x\geq c(m,M,n)>0$, which provides a uniform lower bound for all $t\in[0,T]$ and $x\in[0,L]$ 

Because of this, we can use Theorem \ref{tm:interncontrgencon} to conclude the proof. 
\end{proof}

\subsection{Controlling two target points}
\paragraph{}
In this section, we prove the following result:

\begin{theorem}\label{tm:contrajhbis}
Consider the system \eqref{con:heat0genbis}. Also, consider two analytic functions $h_1,h_2$ such that $h_1(t)<h_2(t)$ for all $t\in[0,T]$, with $\min_{[0,T]}h_1(t)>0$ and $\max_{[0,T]}h_2(t)<L$.
Then, for all $w_1,w_2\in L^2(0,T)$ and $\eps>0$ there exist two controls $v_0\in L^2(0,T)$ and $v_L\in L^2(0,T)$ such that:
\[\|y(t,h_1(t))-w_1(t)\|_{L^2(0,T)}+\|y(t,h_2(t))-w_2(t)\|_{L^2(0,T)}<\eps,\]
where $y$ is the solution of \eqref{con:heat0genbis}. 
\end{theorem}
\begin{proof}
Because of the change of variables defined in Theorem \ref{tm:contrajh}, we may assume that $L=1$, and that $h_1=k$ for some $k>0$. For the sake of simplifying the notation, we denote $h_2$ by $h$. Also, we denote $$\displaystyle{m=\min_{t\in [0,T]}h(t)},$$ which satisfies $m>k$. 
Again, we expect to reduce the problem to the setting of Theorem \ref{tm:interncontrgencon} by means of an appropriate change of variables:
\[\chi(t,x)=\alpha(t)x^{n}+\beta(t)x^{r}+\gamma(t)x.
\]
Here $n$ and $r$ are sufficiently large constants to be defined later on. 
We are going to find $\alpha$, $\beta$, $\gamma$, $n$ and $r$ such that:
\begin{itemize}
    \item $\chi(t,0)=0$,
    \item $\chi(t,k)=\tilde k$,
     \item $\chi(t,h(t))=\tilde m$,
    \item $\chi(t,1)=1$,
     \item There is $c>0$ such that   $\chi_x(t,x)>c$ 
    for all $t\in[0,T]$ and $x\in[0,L]$. 
\end{itemize}
The values $\tilde k$ and $\tilde m$ are small enough constants which will be determined later on.  
For that, we need to solve the following linear system for all $t\in[0,T]$:
\begin{equation*}
    \begin{cases}
        \alpha(t)+\beta(t)+\gamma(t)=1,\\
        \alpha(t)k^{n}+\beta(t)k^r+\gamma(t)k=\tilde k,\\
        \alpha(t)(h(t))^{n}+\beta(t)(h(t))^r+\gamma(t)h(t)=\tilde m.
    \end{cases}
\end{equation*}
Again, the values $r$ and $n$ will be determined later on, with $n$ considerably larger than $r$. The selection of the parameters will be done at the end of the proof, showing that there is no circular fallacy.

In order to solve the system, we are going to use the Cramer's rule. First, note that:
\begin{equation}\label{eq:maindet}
\begin{split}
\begin{vmatrix}
1&1&1\\k^{n}&k^{r}&k\\(h(t))^{n}&(h(t))^{r}&h(t)
\end{vmatrix}
&=kh(t)(k^{r-1}-(h(t))^{r-1})+O(k^n+(h(t))^n).
\end{split}
\end{equation}
Considering that $h(t)>k$, as $k=h_1<h_2=h$, this determinant is negative for $n$ large enough.
Let us now obtain the coefficients in the numerators of the Cramer formula:
\begin{itemize}
\item Let us first start by obtaining the determinant corresponding to the coefficient of $\alpha(t)$:
\begin{equation}\label{eq:alphaneg}
\begin{split}
\begin{vmatrix}
1&1&1\\\tilde k&k^r&k\\\tilde m&(h(t))^{r}&h(t)
\end{vmatrix}
=kh(t)(k^{r-1}-(h(t))^{r-1})+O(\tilde k+\tilde m),
\end{split}
\end{equation}
which is negative  as long as $\tilde k$
and $\tilde m$ are sufficiently small depending on $r$, as $\min_{t\in[0,T]}h(t)<K$.

\item Let us continue with the coefficient corresponding to $\beta(t)$:
\begin{equation}\label{eq:betaneg}
\begin{split}
\begin{vmatrix}
1&1&1\\k^{n}&\tilde k&k\\(h(t))^{n}&\tilde m&h(t)
\end{vmatrix}&
=\tilde kh(t)-k\tilde m+O(k^n+(h(t))^n)
,
\end{split}
\end{equation}
which is negative for $n$ large enough if:
\begin{equation}\label{est:betaneg}
\frac{\tilde k}{\tilde m}<\frac{k}{\max_{[0,T]}h(t)}.
\end{equation}

\item Let us conclude with the coefficient corresponding to $\gamma(t)$:
\begin{equation}\label{eq:gammaneg}
\begin{split}
\begin{vmatrix}
1&1&1\\k^{n}&k^r&\tilde k\\(h(t))^{n}&(h(t))^{r}&\tilde m
\end{vmatrix}
&=k^r\tilde m-\tilde k(h(t))^r+O(k^n+(h(t))^n),
\end{split}
\end{equation}
which is negative for $n$ large enough if:
\begin{equation}\label{est:gammaneg}
\frac{k^r}{(\min_{[0,T]}h(t))^r}<\frac{\tilde k}{\tilde m}.
\end{equation}
So, we need to choose $\tilde k$ and $\tilde m$ so that we have, at the same time, \eqref{est:betaneg} and \eqref{est:gammaneg}. 
\end{itemize}

Summing up, in order to show that we are not falling into a circular fallacy, we choose the parameters as follows:
\begin{itemize}
    \item First, we fix the ratio $\delta=\frac{\tilde k}{\tilde m}$ so that:
    \[\delta<\frac{k}{\max_{[0,T]}h(t)}.\]
    \item Next, we fix $r>0$ large enough so that 
    \begin{equation*}
\frac{k^r}{(\min_{[0,T]}h(t))^r}<\delta.
\end{equation*}
This can be done as $k<\min_{t\in[0,T]}h(t)$.
    \item Next, we fix $\tilde m$ small enough and fix $\tilde k=\tilde m\delta$ 
    so that the value in \eqref{eq:alphaneg} is strictly negative.
    \item Finally, we fix $n$ large enough so that the values in \eqref{eq:maindet}, \eqref{eq:betaneg} and \eqref{eq:gammaneg} are strictly negative. 
\end{itemize}
With that choice, we ensure that all the coefficients of $\chi$ are strictly positive, so, in particular: 
\[\chi_x(t,x)\geq c_{h}=\min_{s\in [0,T]}\gamma(s)>0
\qquad\forall t\in[0,T]\quad\forall x\in[0,L].
\]
Thus, with $\chi$, we may conclude applying the change of variable as in the proof of Theorem \ref{tm:contrajh}.
\end{proof}

\section{Numerical results}\label{sec:num}

This section illustrates the constructive duality procedure of Lemma~\ref{lm:interiordualgen} through a series of numerical experiments. We compute boundary controls that approximately enforce prescribed time signals at interior observation point(s), considering both fixed-point and moving-point configurations. We first present constant-coefficient heat equation examples with one boundary control and one tracked point, and then with two boundary controls and two tracked points. We also include a variable-coefficient parabolic example, as well as a moving-target experiment.

\subsection{Discretization and optimization strategy}
We consider a uniform partition of $[0,L]$ with $N_e$ subintervals and use continuous and piecewise affine $\mathbb{P}_1$ finite elements in space~\cite{zienkiewicz1977finite}. In time, we discretize by the backward Euler method with step size $\Delta t = T/N_t$, where $N_t$ denotes the number of time steps. In all the simulations of this section, we take a uniform mesh with $N_e = 200$ elements and $\Delta t = 10^{-3}$.

We first describe the discretization of the state problem \eqref{con:heat0gen}-\eqref{con:heat0genbis}. After a standard lifting of the Dirichlet boundary data and elimination of the boundary degrees of freedom, the vector $y^n$ of interior nodal values at time $t_n = n\Delta t$ satisfies a linear system of the form
\[
\left(\frac{1}{\Delta t}M + A^n\right)y^n
=
\frac{1}{\Delta t}My^{n-1} + G^n,
\qquad n=1,\dots,N_t,
\]
where $M$ is the mass matrix, $A^n$ is the matrix associated with the spatial operator $\mathcal{L}$ (assembled at time $t_n$ when the coefficients depend on time), and $G^n$ collects the contributions of the boundary controls. The state trace $y(t_n,x_i)$ is then approximated by $\mathbb{P}_1$ interpolation from the nodal values.

Since the adjoint equation \eqref{eq:dualsyspointcongen} is posed backward in time, we introduce the new variable $s=T-t$ and define
\[
q(s,x):=p(T-s,x).
\]
In this way, we solve an equivalent forward parabolic problem with homogeneous Dirichlet boundary conditions. After eliminating the boundary degrees of freedom, the vector $q^n$ of interior nodal values satisfies
\[
\left(\frac{1}{\Delta t}M + A^{*,n}\right)q^n
=
\frac{1}{\Delta t}Mq^{n-1} + \sum_{i=1}^N f_i^n\, b_i,
\qquad n=1,\dots,N_t,
\]
where $A^{*,n}$ is the matrix associated with the adjoint operator $\mathcal{L}^*$, and $b_i$ is the discrete load vector associated with the point source located at $x_i$. In the constant-coefficient heat equation examples, this reduces to the usual matrix $\frac{1}{\Delta t}M + aK$, with $K$ the stiffness matrix.

To define $b_{i}$ for a target point $x_i\in(0,L)$, we consider an arbitrary test function $\varphi\in V_h$, where $V_h$ denotes the $\mathbb{P}_1$ finite element space. The point source is introduced in weak form by
\[
\langle \delta_{x_i},\varphi\rangle_{(C^0([0,L])'\times C^0([0,L]))} = \varphi(x_i).
\]
If $\{\phi_j\}_{j=1}^{N_e+1}$ denotes the nodal basis of $V_h$, then
\[
(b_{i})_j=\langle \delta_{x_i},\phi_j\rangle_{(C^0([0,L])'\times C^0([0,L]))}=\phi_j(x_i).
\]

Hence, the vector $b_i$ is supported on the two endpoints of the element containing $x_i$. The same interpolation principle is used to evaluate the state trace at the target points.

In addition, to recover the boundary controls from the dual solution, the boundary derivatives $\partial_x p(t,0)$ and $\partial_x p(t,L)$ are approximated by one-sided finite differences based on the first and last interior finite element nodal values, in a manner consistent with the $\mathbb{P}_1$ discretization.

Finally, we discretize the functional \eqref{eq:funcminimgen} (and its two-control analogue, \eqref{eq:funcminimgen2}) on the time grid. The unknowns are the discrete vectors $(f_i^n)_{n=1}^{N_t}$, that is, $f_i\in\mathbb{R}^{N_t}$ for each tracked point, which enter the discrete adjoint dynamics through the vectors $b_i$. To ensure differentiability of the regularization term, we use the smoothed norm
\[
\|(f_1,\dots,f_N)\|_{(L^2(0,T))^N}
\approx
\sqrt{\Delta t\sum_{n=1}^{N_t}\sum_{i=1}^N |f_i^n|^2 + \delta},
\]
with $\delta>0$ small. In all the examples, we take $\delta=10^{-14}$. The minimization is carried out with a quasi-Newton method (BFGS, as implemented in \texttt{fminunc}), using an analytically assembled gradient computed through a discrete adjoint sweep. In the examples for which the backward Euler matrix is time-independent, a sparse Cholesky factorization is performed once and reused at every iteration.

\subsection{Example 1: One boundary control and one tracked point for the heat equation}

We consider system \eqref{con:heat0gen} on $(0,T)\times(0,L)$ with $L=1$, $a=1$, $b=c=0$ and $T=0.5$:
\[
\begin{cases}
y_t - y_{xx}=0, & (t,x)\in(0,T)\times(0,1),\\
y(t,0)=0,\quad y(t,1)=v(t), & t\in(0,T),\\
y(0,x)=0, & x\in(0,1).
\end{cases}
\]
We track the interior trace at $x_1=0.5$ towards a prescribed target $w\in L^2(0,T)$:
\[
y(t,x_1)\approx w(t)\quad\text{in }L^2(0,T).
\]
The target is chosen as a sinusoid,
\[
w(t)=A\sin\Big(\frac{2\pi m}{T}t\Big),
\]
with amplitude $A=1$ and $m=2$ oscillations over $[0,T]$.
The dual functional \eqref{eq:funcminimgen} is minimized with regularization parameters $\varepsilon=10^{-1}$ and $\varepsilon=10^{-2}$, producing an optimizer $f$ and therefore an adjoint state $p$.
The control is then recovered as $v(t)=\partial_x p(t,1)$ and injected into the forward solver.

Figure~\ref{fig:ex1_track} compares the achieved trace $y(\cdot,x_1)$ with the target $w$ and Figure~\ref{fig:ex1_control} displays the computed control $v$ for each parameter.
The final tracking mismatch is measured by
\[
E_1 := \|y(\cdot,x_1)-w(\cdot)\|_{L^2(0,T)}.
\]
We obtain the tracking mismatch $E_1=1.000108\times 10^{-1}$ for $\varepsilon=10^{-1}$ and $E_1=1.000862\times10^{-2}$ for $\varepsilon=10^{-2}$.

\begin{figure}[ht]
  \centering
\includegraphics[width=0.48\textwidth]{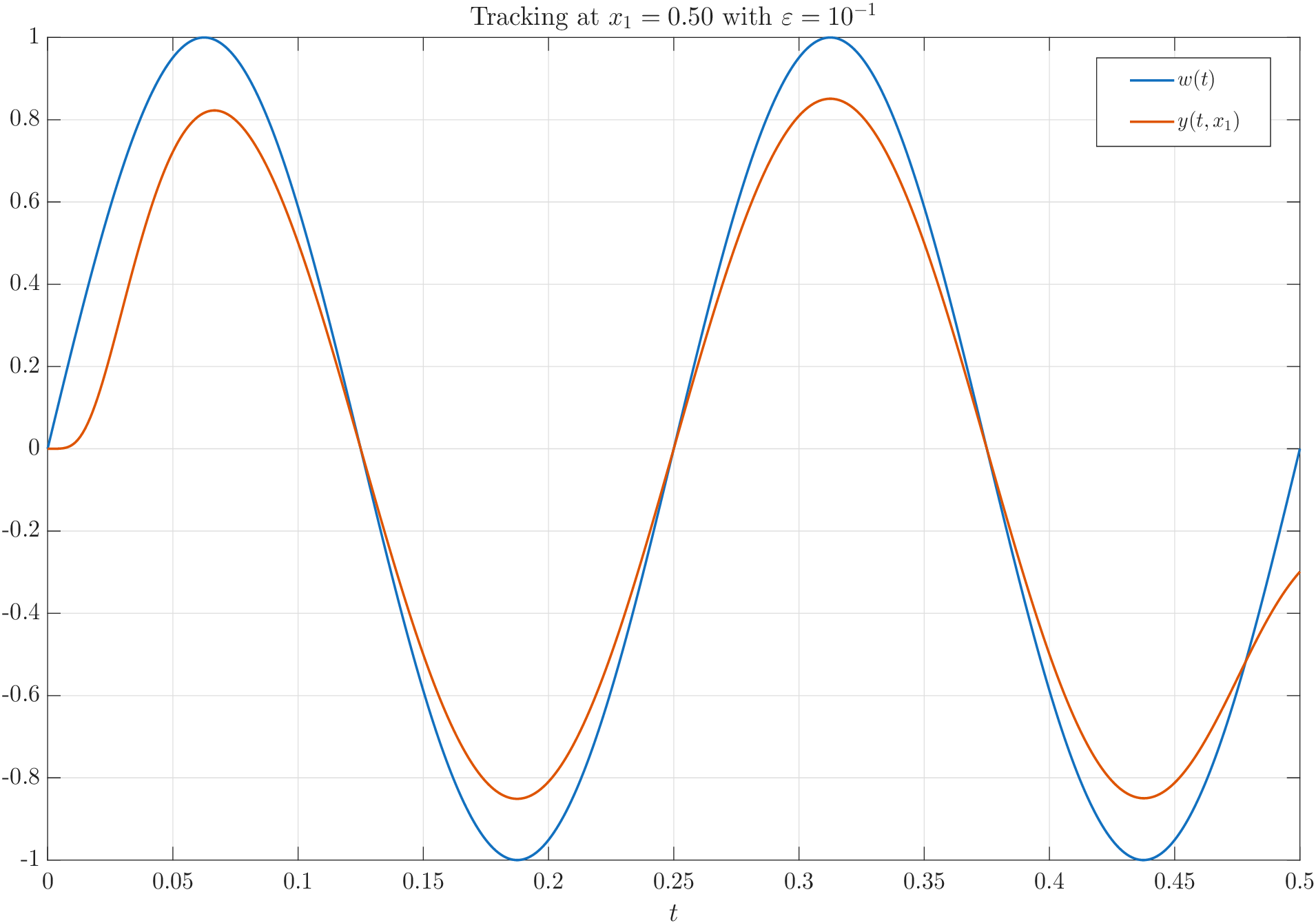}
\includegraphics[width=0.48\textwidth]{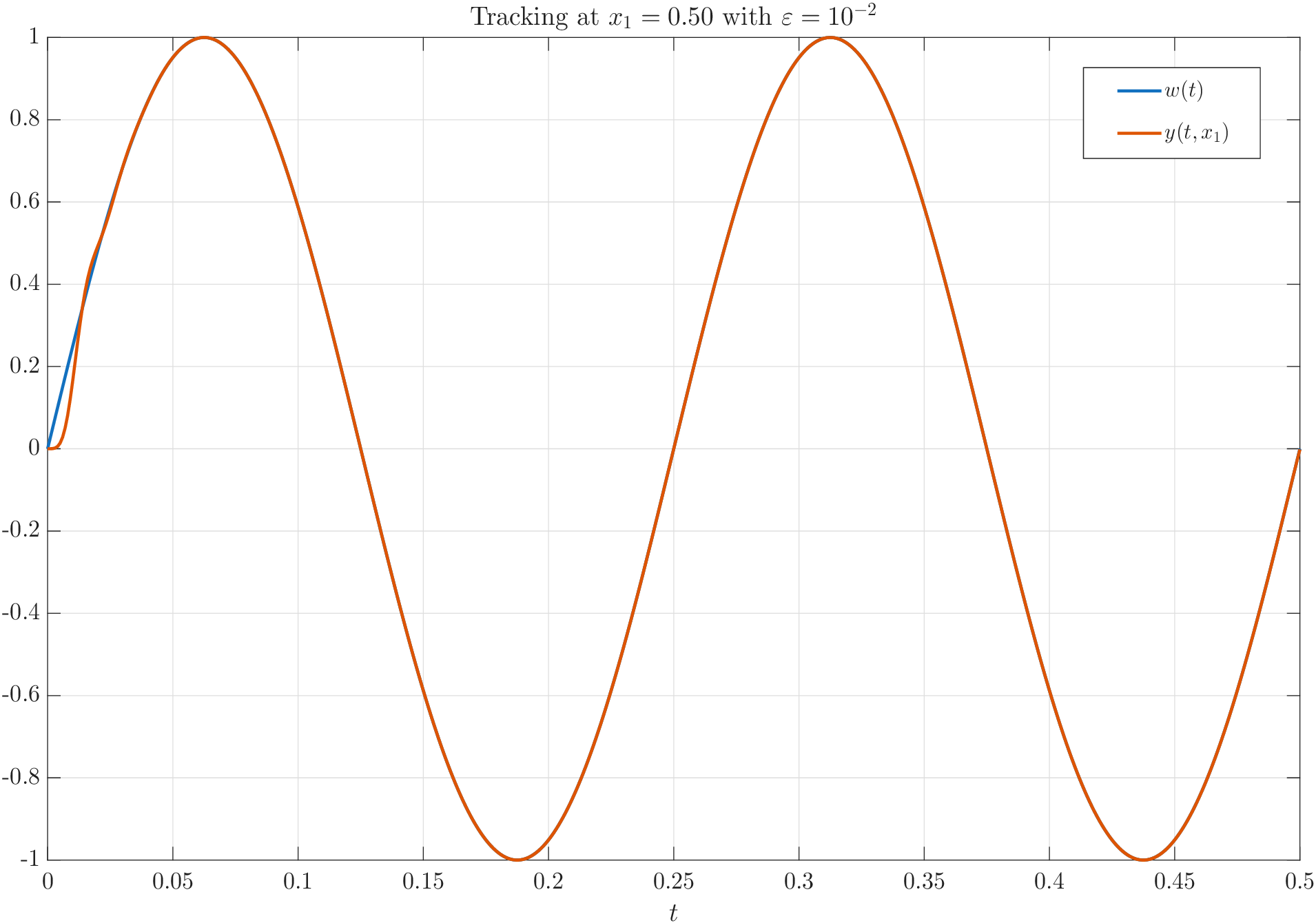}
  \caption{Target $w(t)$ and achieved trace $y(t,x_1)$ at $x_1=0.5$ for $\varepsilon=10^{-1}$ (left) and $\varepsilon=10^{-2}$ (right).}
  \label{fig:ex1_track}
\end{figure}

\begin{figure}[ht]
  \centering
\includegraphics[width=0.48\textwidth]{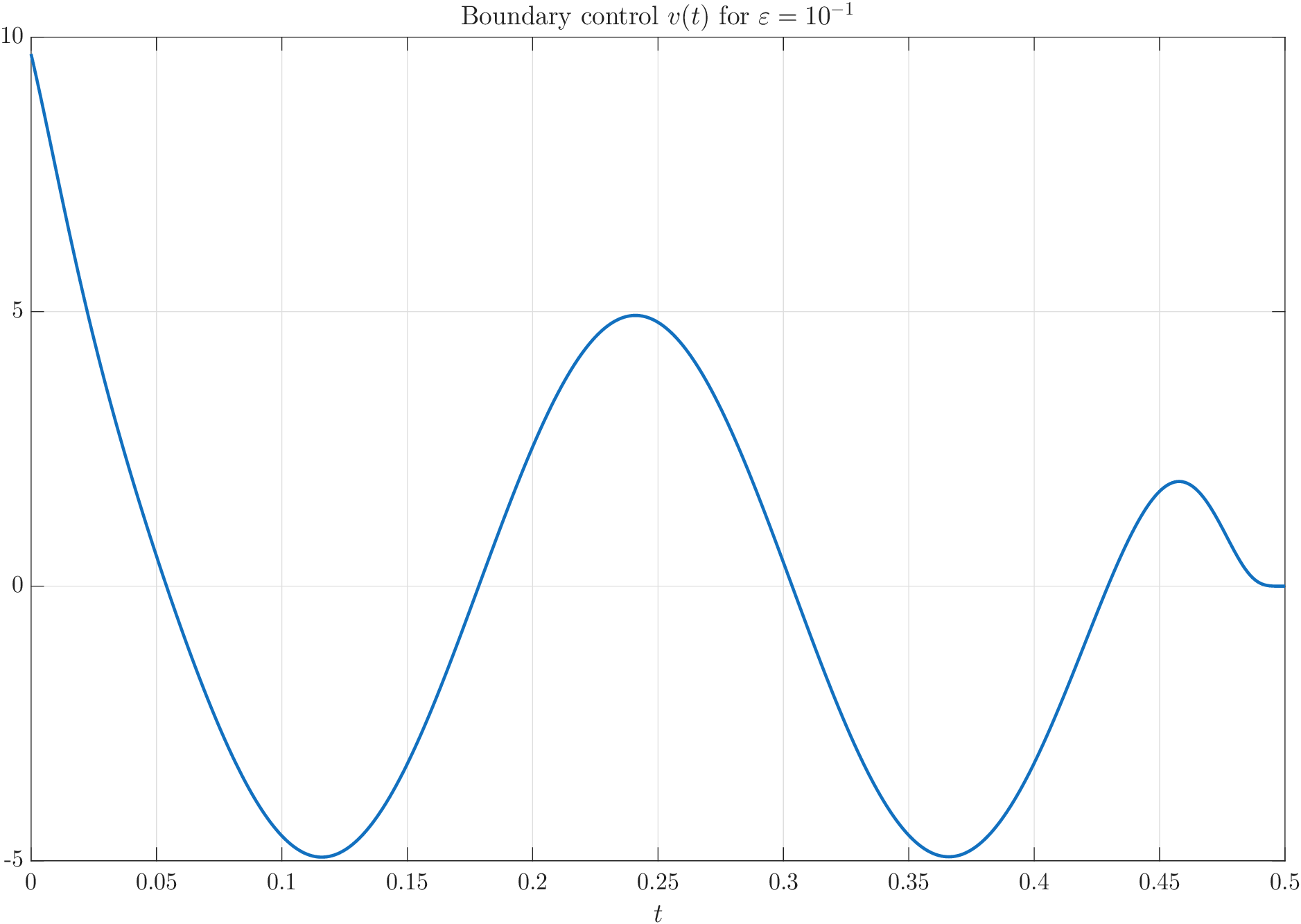}
\includegraphics[width=0.48\textwidth]{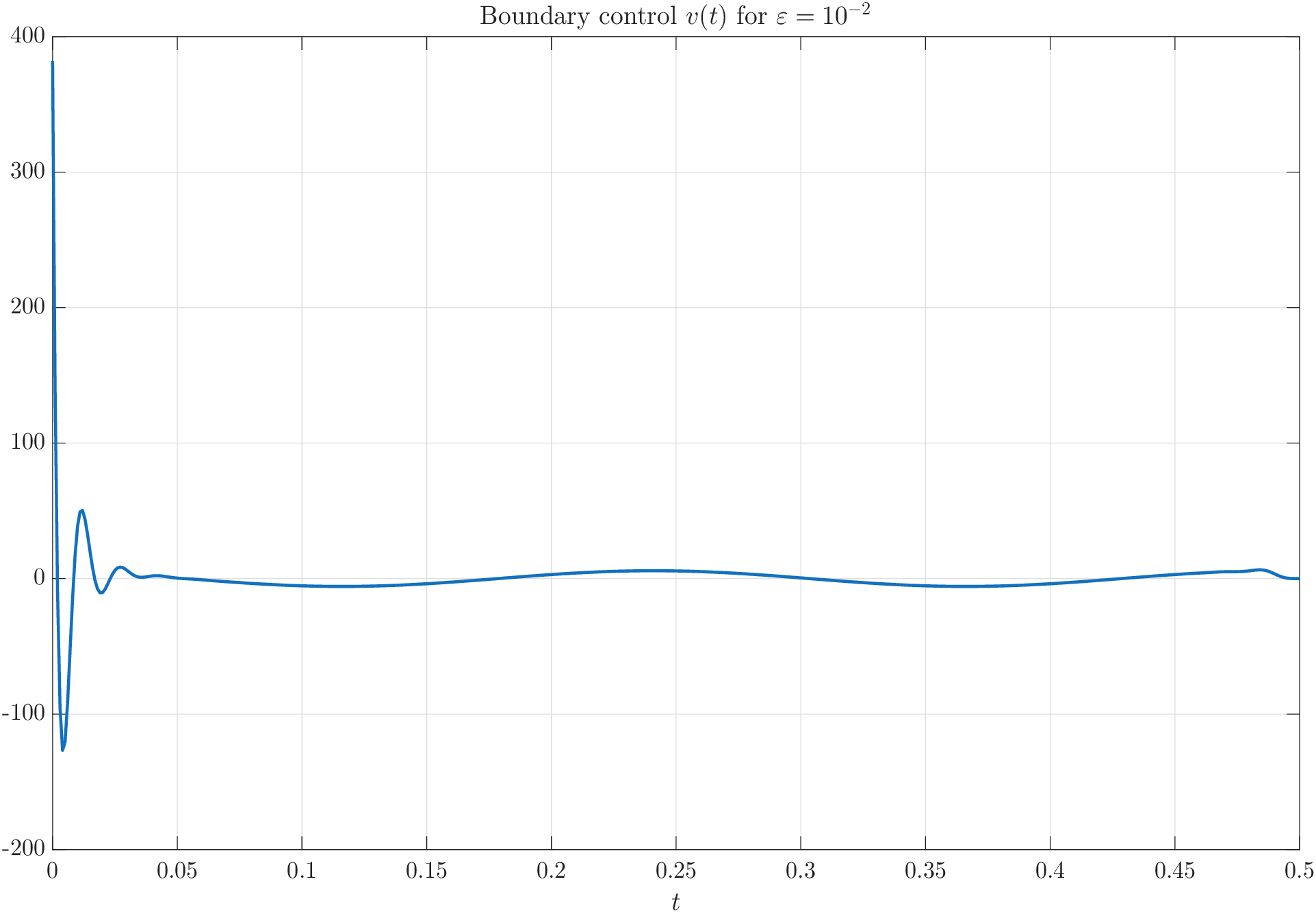}
  \caption{Boundary control $v(t)=\partial_x p(t,1)$ for $\varepsilon=10^{-1}$ (left) and $\varepsilon=10^{-2}$ (right).}
  \label{fig:ex1_control}
\end{figure}

\subsection{Example 2: Two boundary controls and two tracked points for the heat equation}

We next consider system \eqref{con:heat0genbis} with two Dirichlet controls and $T=1$:
\[
\begin{cases}
y_t - y_{xx}=0, & (t,x)\in(0,T)\times(0,1),\\
y(t,0)=v_0(t),\quad y(t,1)=v_L(t), & t\in(0,T),\\
y(0,x)=0, & x\in(0,1).
\end{cases}
\]
We aim to track two interior traces simultaneously:
\[
y(t,x_1)\approx w_1(t),\qquad y(t,x_2)\approx w_2(t),
\]
with $x_1=0.25$ and $x_2=0.5$.
Following Lemma~\ref{lm:interiordualgen} (two-control version), we introduce two dual forcing terms $f_1,f_2\in L^2(0,T)$ in \eqref{eq:dualsyspointcongen}, and minimize the corresponding dual functional featuring both boundary fluxes $\partial_x p(t,0)$ and $\partial_x p(t,1)$ with $\varepsilon=10^{-3}$.
The controls are then recovered as
\[
v_0(t)=\partial_x p(t,0),\qquad v_L(t)=\partial_x p(t,1).
\]

The targets are chosen as smooth ramp functions (with $k=1$):
\[
w_1(t)= t\big(1-e^{-kt}\big),\qquad
w_2(t)= \tfrac12\, t\big(1-e^{-kt}\big).
\]

Figure~\ref{fig:ex2_tracks} shows the comparison between the targets and the achieved traces, and Figures~\ref{fig:ex2_controls} displays the computed boundary controls.
We quantify the errors by
\[
E_1 := \|y(\cdot,x_1)-w_1(\cdot)\|_{L^2(0,T)},\qquad
E_2 := \|y(\cdot,x_2)-w_2(\cdot)\|_{L^2(0,T)}.
\]
We obtain the tracking mismatches $E_1=1.414322\times 10^{-3}$ and $E_2=3.270044\times 10^{-4}.$

\begin{figure}[ht]
  \centering
\includegraphics[width=0.48\textwidth]{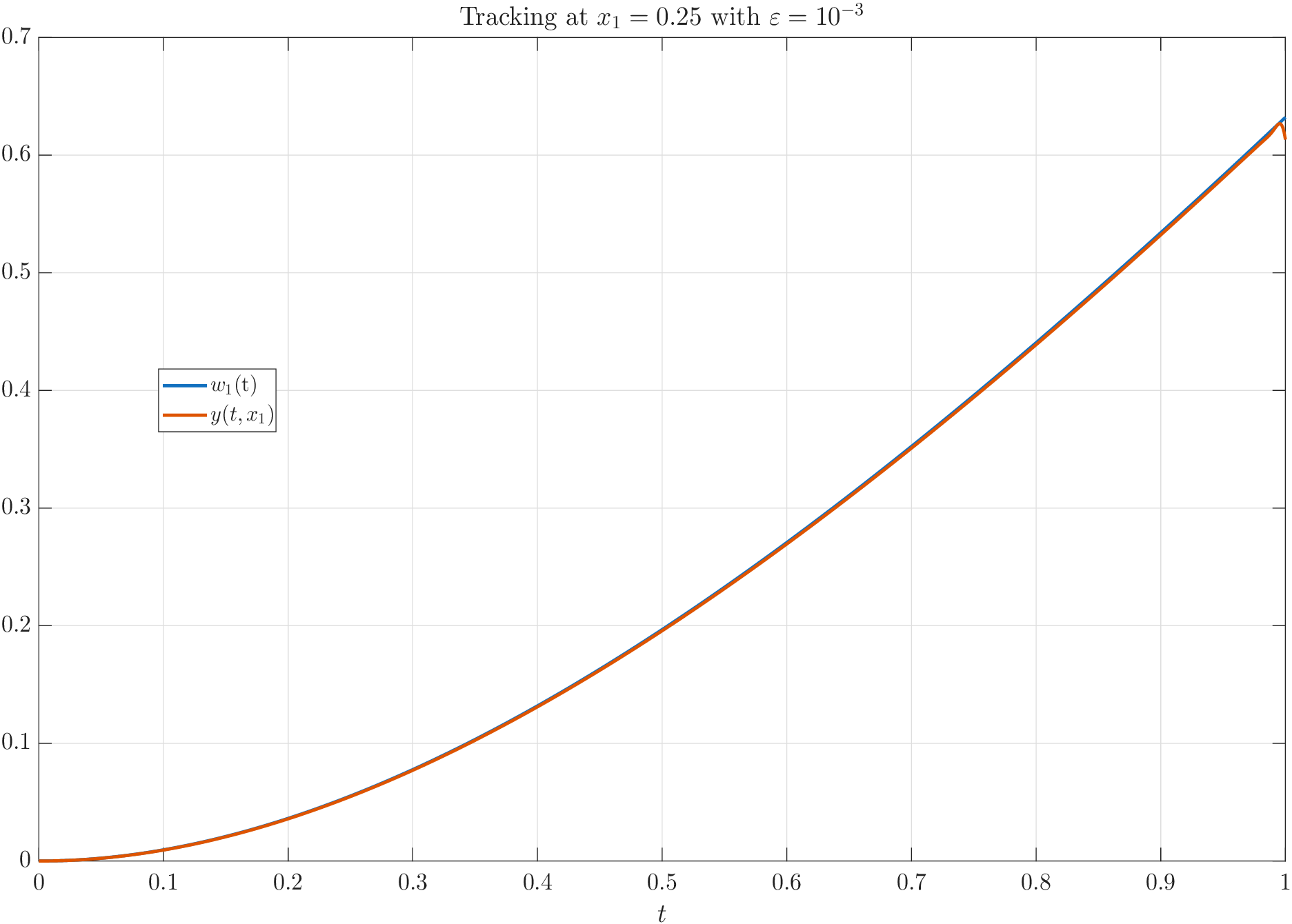}
\includegraphics[width=0.48\textwidth]{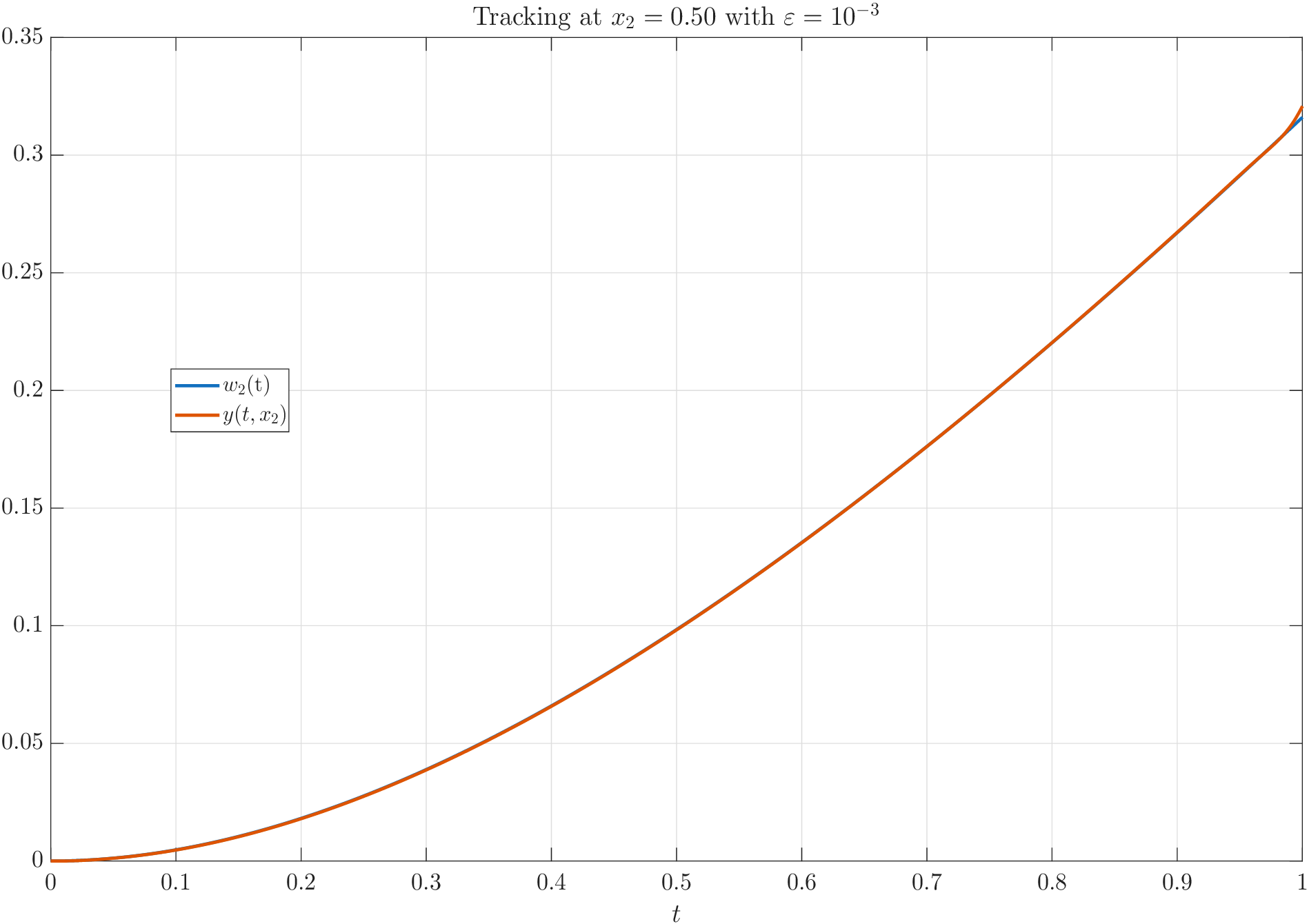}
  \caption{Target $w_1(t)$ and achieved trace $y(t,x_1)$ at $x_1=0.25$ (left), and target $w_2(t)$ and achieved trace $y(t,x_2)$ at $x_2=0.5$ (right), for $\varepsilon=10^{-3}$.}
  \label{fig:ex2_tracks}
\end{figure}

\begin{figure}[ht]
  \centering
\includegraphics[width=0.47\textwidth]{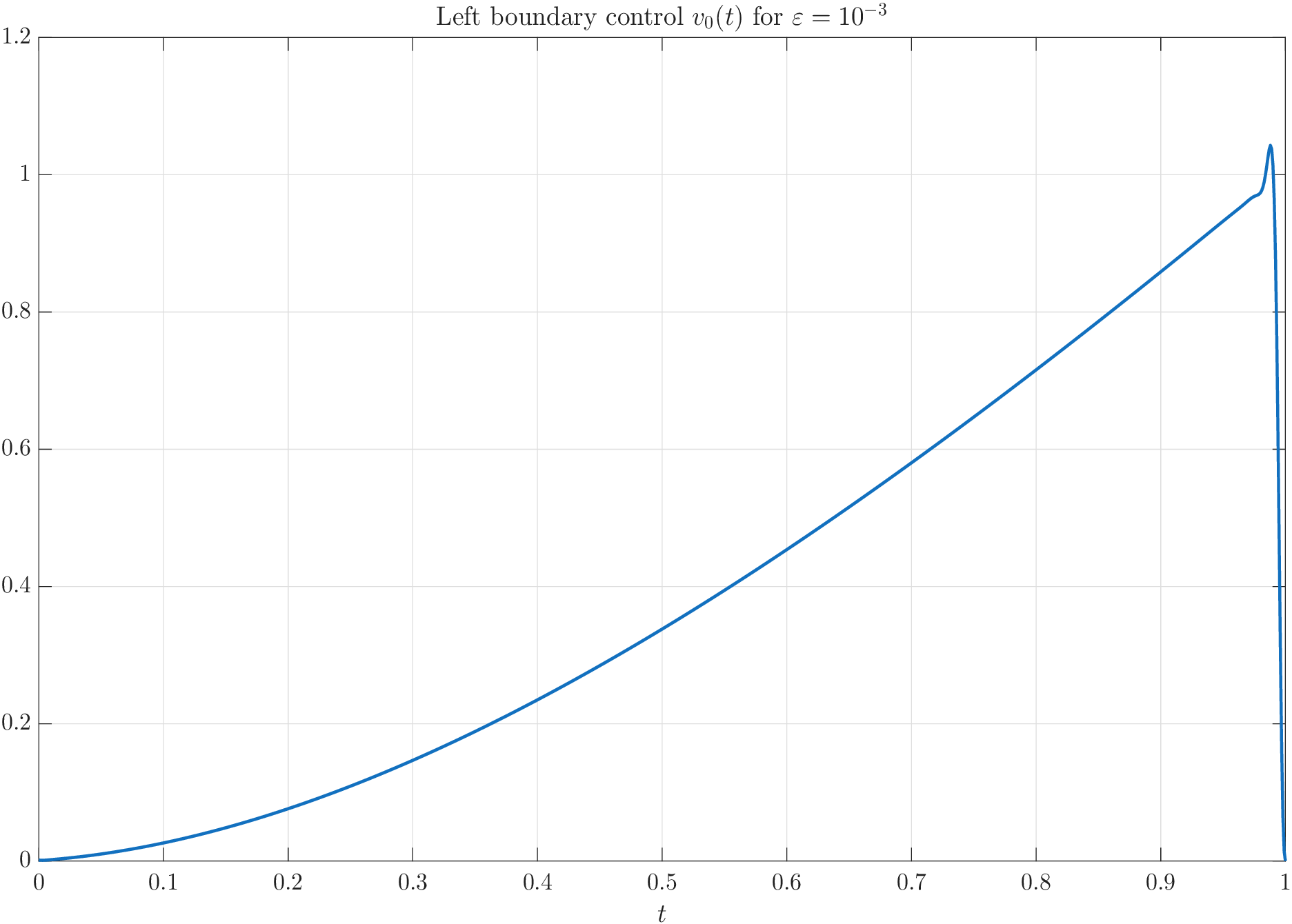}
\includegraphics[width=0.47\textwidth]{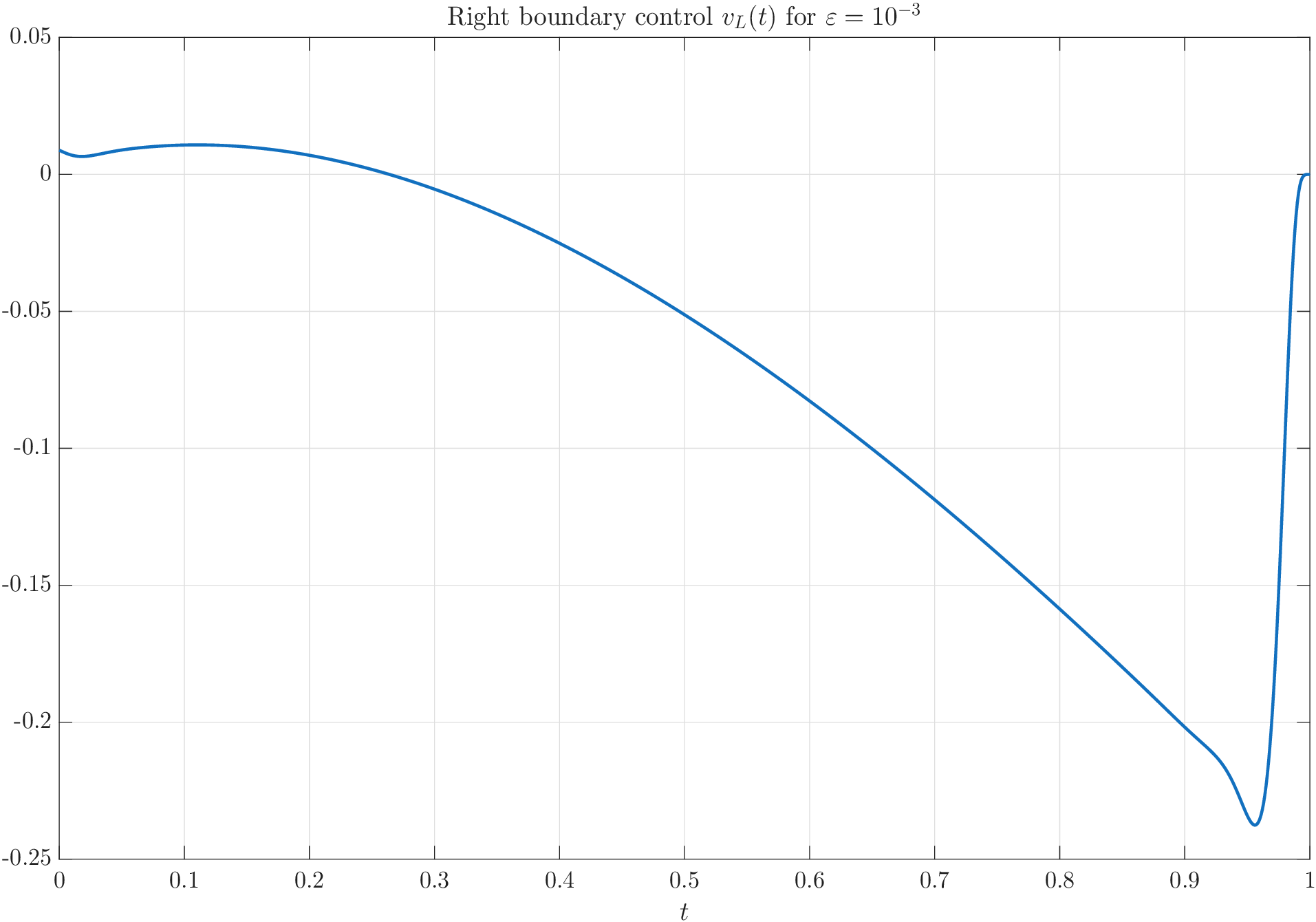}
  \caption{Boundary control $v_0(t)=\partial_x p(t,0)$ (left) and boundary control $v_L(t)=\partial_x p(t,1)$ (right), for $\varepsilon=10^{-3}$.}
  \label{fig:ex2_controls}
\end{figure}

\subsection{Example 3: One boundary control and one tracked point for variable coefficients}

We now illustrate the same duality-based procedure for a parabolic equation with nonconstant coefficients. We consider system \eqref{con:heat0gen} on $(0,T)\times(0,L)$ with $L=1$ and $T=0.5$:
\[
\begin{cases}
y_t-a(x)y_{xx}+b(x)y_x+c(t)y=0, & (t,x)\in (0,T)\times(0,1),\\[1mm]
y(t,0)=0,\qquad y(t,1)=v(t), & t\in(0,T),\\[1mm]
y(0,x)=0, & x\in(0,1).
\end{cases}
\]
The coefficients are chosen as
\[
a(x)=1+0.15\cos(\pi x),\qquad
b(x)=0.1\sin(\pi x),\qquad
c(t)=0.3(1+t).
\]
We track the interior trace at the point $x_1=0.75$ towards the target:
\[
w(t)=\sin^2\!\left(\frac{\pi t}{T}\right).
\]

At the discrete level, the second-order part is rewritten in divergence form as
\[
-a(x)y_{xx}+b(x)y_x=-(a(x)y_x)_x+\beta(x)y_x,
\qquad \beta(x)=a_x(x)+b(x),
\]
so that the corresponding weak bilinear form becomes
\[
\int_0^L a(x)\,y_x\varphi_x\,dx
+\int_0^L \beta(x)\,y_x\varphi\,dx
+\int_0^L c(t)\,y\varphi\,dx.
\]
The dual functional is minimized with regularization parameter $\varepsilon=10^{-3}$.

Unlike in the constant-coefficient case, the reaction term $c(t)$ makes the discrete parabolic operator time-dependent, so the backward Euler matrix is assembled and factorized at each time step.\\

The tracking mismatch is measured by $
E_3:=\|y(\cdot,x_1)-w(\cdot)\|_{L^2(0,T)},$ and we obtain
\[
E_3 = 9.481350\times 10^{-4}.
\]

Figure~\ref{fig:ex3_tracking_control} displays the comparison between the target $w$ and the achieved trace $y(\cdot,x_1)$, and the boundary control $v$.

\begin{figure}[ht]
  \centering
\includegraphics[width=0.47\textwidth]{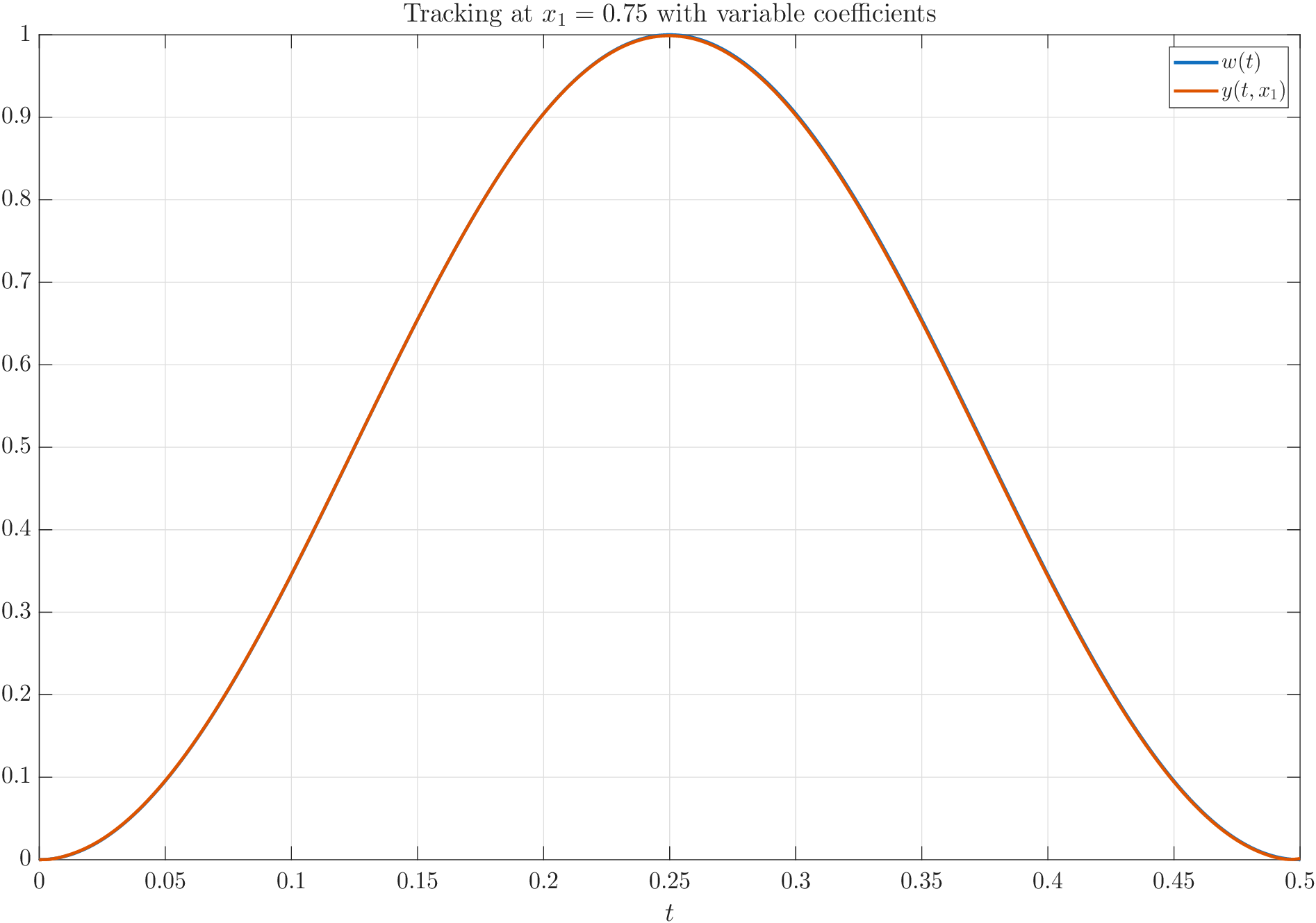}
\includegraphics[width=0.47\textwidth]{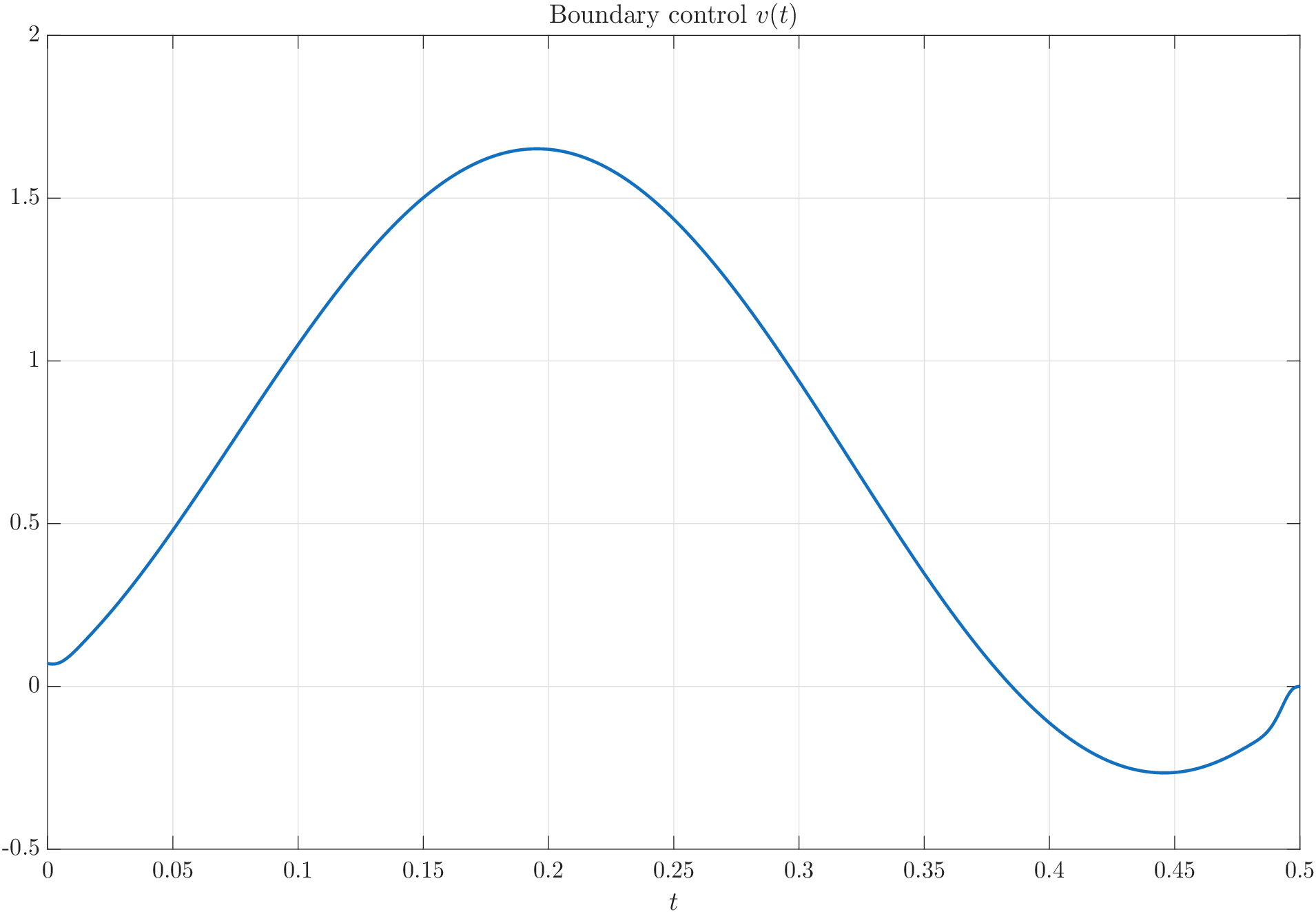}
  \caption{Target $w(t)$ and achieved trace $y(t,x_1)$ at $x_1=0.75$ (left) and boundary control $v(t)=\partial_x p(t,L)$ (right) for $\varepsilon=10^{-3}$.}
  \label{fig:ex3_tracking_control}
\end{figure}

\subsection{Example 4: One boundary control and one moving tracked point}

As a final test, we illustrate numerically the moving-target setting considered in Section~\ref{sec:mtrg}. 
More precisely, we consider system \eqref{con:heat0gen} on $(0,T)\times(0,L)$ with $L=1$, $T=0.5$, and constant coefficients $a=1$, $b=0$, and $c=0$, that is,
\[
\begin{cases}
y_t-y_{xx}=0, & (t,x)\in(0,T)\times(0,1),\\[1mm]
y(t,0)=0,\qquad y(t,1)=v(t), & t\in(0,T),\\[1mm]
y(0,x)=0, & x\in(0,1).
\end{cases}
\]
In contrast with Examples~1--3, the observation point is now time-dependent. We choose the analytic trajectory
\[
h(t)=0.5+0.15\sin\!\left(\frac{\pi t}{T}\right),
\]
which satisfies
\[
0.5\le h(t)\le 0.65\qquad \text{for all } t\in[0,T].
\]
The target function is selected as
\[
 w(t)= e^{-\frac{(t-t_0)^2}{2\sigma^2}}
\]
with $t_0 = T/2,$ and $\sigma = T/16$.
Our goal is therefore to enforce
\[
y(t,h(t))\approx w(t),\;\;\text{in }L^2(0,T).
\]

The dual minimization is carried out exactly as in Example~1, with regularization parameter $\varepsilon=10^{-3}$. The only modification is that the state trace is no longer evaluated at a fixed point $x_1$, but at the moving location $h(t_n)$ at each time step, using the same $\mathbb{P}_1$ interpolation procedure as in the fixed-point case.

The tracking error is measured by
$E_4:=\|y(\cdot,h(\cdot))-w(\cdot)\|_{L^2(0,T)}$, and we obtain
$$E_4=1.000641\times 10^{-3}.$$
In the computations, we observe that the achieved trace $y(t,h(t))$ follows the prescribed target $w(t)$ accurately, confirming at the numerical level the moving-point controllability result proved in Theorem \ref{tm:contrajh}.

Figure~\ref{fig:ex4_tracking_control} compares the target $w(t)$ with the computed trace $y(t,h(t))$ and shows the corresponding boundary control $v$. Finally, Figure~\ref{fig:ex4_state_adjoint} shows the space-time distributions of the state $y(x,t)$ and the adjoint $p(x,t)$. The dashed white curve indicates the moving trajectory $x=h(t)$.

\begin{figure}[ht]
  \centering
\includegraphics[width=0.47\textwidth]{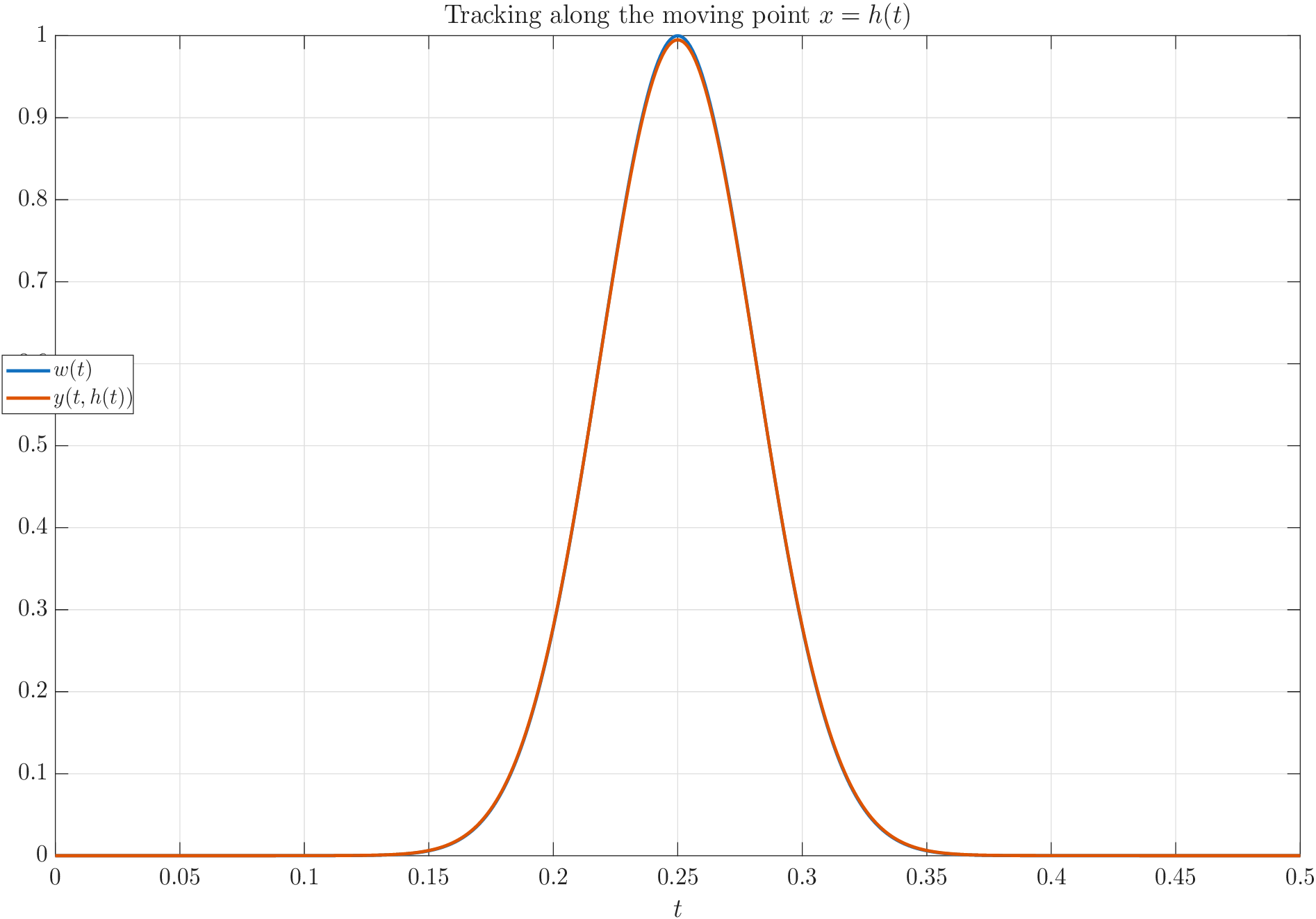}
\includegraphics[width=0.47\textwidth]{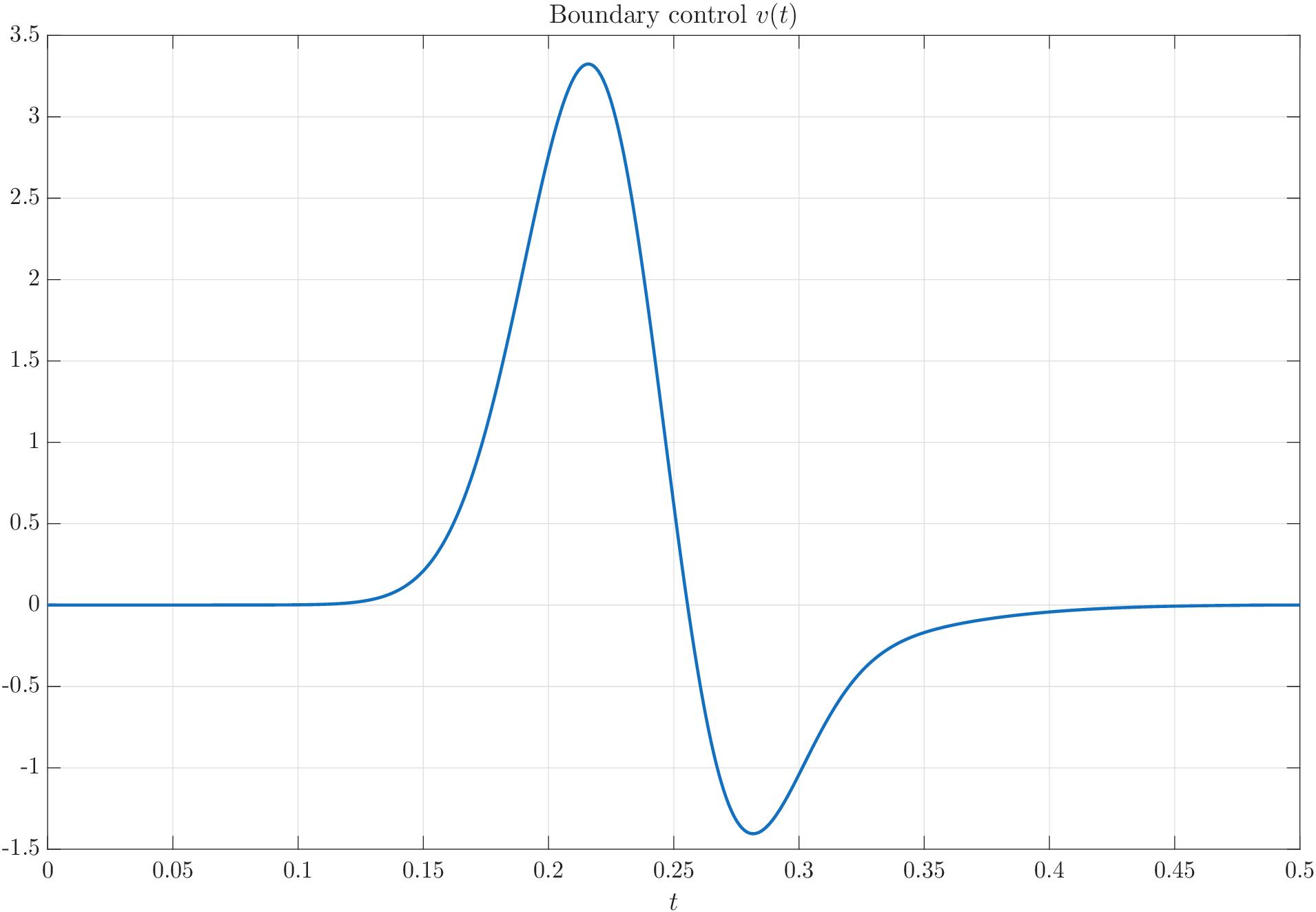}
  \caption{Target $w(t)$ and achieved trace $y(t,h(t))$ (left) and boundary control $v(t)$ (right), for $\varepsilon=10^{-3}$.}
  \label{fig:ex4_tracking_control}
\end{figure}

\begin{figure}[ht]
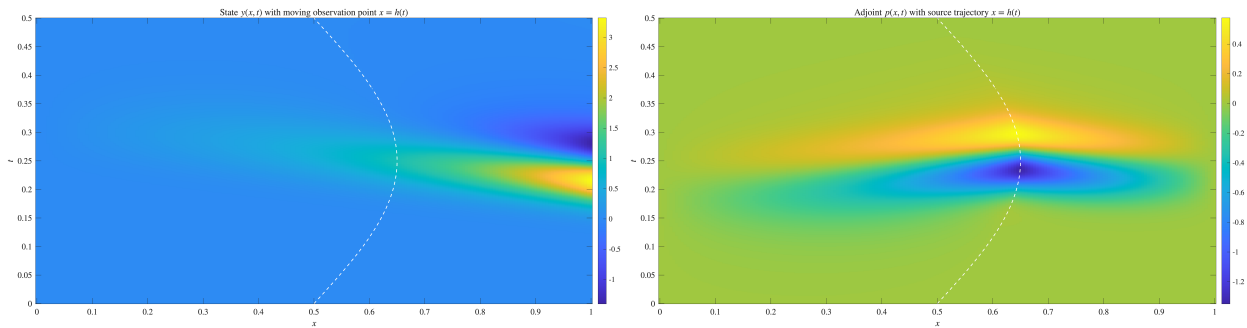

  \centering
\includegraphics[width=0.48\textwidth]{Example4_state.png}
\includegraphics[width=0.48\textwidth]{Example4_adjoint.png}
  \caption{Space-time plots of the state $y(x,t)$ (left), the adjoint $p(x,t)$ (right), and the moving trajectory $x=h(t)$, for $\varepsilon=10^{-3}$.}
  \label{fig:ex4_state_adjoint}
\end{figure}

\section{Open problems and possible extensions}\label{sec:op}
In this section, we would like to provide some additional remarks and propose some relevant open problems:
\begin{itemize}
\item \textbf{Analytic form of controls.} Under very specific setting, we can provide explicit formulas for the controls. This can be done with the flatness approach (see \cite{fliess1995flatness} or \cite{Levine2009}, for example, for the explanation and references of this method in Control Theory). Consider the following system:
\begin{equation}\label{con:heat0bis}
\begin{cases}
y_{t}-\partial_{xx} y=0 & \mbox{ in } (0,T)\times(0,L),\\
y(\cdot,0)=v_0 &\mbox{ on } (0,T),\\
y(\cdot,L)=v_L& \mbox{ on } (0,T),\\
y(0,\cdot)=0 & \mbox{ on }(0,L).
\end{cases}
\end{equation} 
If $s\in(1,2)$
and $w_1$, $w_2$ are two Gevrey functions of order $s$ (for the definition and a first application of this functions to controllability of parabolic equations, see \cite{laroche1998motion} or \cite{laroche2000motion}) that annihilate  at $t=0$, then there are $v_0$ and $v_L$
Gevrey functions of order $s$  such that 
the solution of \eqref{con:heat0bis} satisfies:
\begin{equation}\label{eq:targsydxy}
y(\cdot,x_1)=w_1 \ \mbox{ and } \  \partial_x y(\cdot,x_1)=w_2.
\end{equation}
In fact, by using the flatness approach and following the procedure used in \cite{laroche2000motion}, if we apply the controls:
\begin{equation*}
\begin{split}
&v_0(t):= \sum_{i\geq0}\frac{w_1^{(i)}(t)}{(2i)!}x_1^{2i} 
+\sum_{i\geq0}\frac{w_2^{(i)}(t)}{(2i+1)!}(-x_1)^{2i+1}, \\
&v_L(t):= \sum_{i\geq0}\frac{w_1^{(i)}(t)}{(2i)!}(L-x_1)^{2i} 
+\sum_{i\geq0}\frac{w_2^{(i)}(t)}{(2i+1)!}(L-x_1)^{2i+1},
\end{split}
\end{equation*}
 the solution of \eqref{con:heat0bis} is given by:
\begin{equation}\label{eq:yw1w2Gev}
y(t,x)=\sum_{i\geq0}\frac{w_1^{(i)}(t)}{(2i)!}(x-x_1)^{2i} 
+\sum_{i\geq0}\frac{w_2^{(i)}(t)}{(2i+1)!}(x-x_1)^{2i+1}.
\end{equation}
The solution \eqref{eq:yw1w2Gev} of \eqref{con:heat0bis} is well-posed by applying \cite[Proposition 1]{martin2014null}, where they use Stirling asymptotic formula. 
Moreover, we clearly have \eqref{eq:targsydxy}. 
In addition, we can easily prove that $y$ satisfies \eqref{con:heat0bis}  (the first three equations are straightforward, and the last one uses that $w_1$ and $w_2$ annihilate on $t=0$). However, this technique is very specific and cannot be generalized to arbitrary parabolic equations in which $a$, $b$ and $c$ on equation \eqref{con:heat0genbis} depend on the time variable. 

In addition, it is not evident how this technique can be adapted either when there is one control or when we are controlling the solution at two points, even with constant coefficients. In this setting, developing new techniques for analytic controls, remains open. 

\item \textbf{Higher dimensions.} Analyzing analogue problems in higher dimensions is far from trivial. Let us consider $\Omega$ a domain in $\mathbb R^d$, for $d\in\mathbb{N}$ and $d\geq 2$. First of all, we would need to determine the proper formulation: would it  be a matter of controlling the solution in specific curves, or would it be a matter of controlling the solution in manifolds of dimension $d-1$ that change over time? In addition, in order to answer those questions, one would need to develop groundbreaking techniques.

\item \textbf{Controllability with random diffusion.} An interesting problem that has been kept out of the scope of this paper is the tracking controllability of the heat equation (or a general parabolic one, similar to those we have discussed in this paper) with random diffusion. In fact, it is known that one might control the average of the heat equation with a random control whenever the probability of the diffusion of being small is almost null (see \cite{lu2016averaged}, \cite{coulson2019average}, and \cite{barcena2021averaged}). Thus, it seems reasonable that we can also control the average of the traces under that setting. The problem, though, remains open.

\item \textbf{Fluid-structure problem.} Controlling the trajectory of structures surrounded by fluid is a problem of high relevance. In order to tackle such problems, a first approach can be to control the trajectory of the punctual mass in fluid-structure problem with a punctual mass, as the system studied in \cite{liu2013single}. For that, new techniques must be developed. 

\item \textbf{E-tracking controllability.} A closely related problem is $E$-tracking controllability, where instead of controlling punctual values, we may seek to control weighted averages on a certain domain. This notion was introduced and developed in  \cite{danhane2025averaged} for abstract systems.
This should not be confused with the averaged tracking controllability, instead of having unknown dynamics, here we may control the averages of the state. A further understanding of this problem remains open. 
\end{itemize}

\color{black}

\appendix

\section{Analytic functions}

First of all, let us recall the definition of analytic function on a closed rectangle:

\begin{definition}\label{def:analytic}
   Let $F:[0,T]\times[0,L]\to\mathbb R$. We say that $F$ is analytic if for every point $(t_0,x_0)\in [0,T]\times[0,L]$, there exist a radius $r>0$ 
   and a sequence of coefficients $(a_{n,m})_{n,m=0}^\infty\subset\mathbb R$
   such that the power series
   \[
   \sum_{n,m\geq0}a_{n,m}(t-t_0)^n(x-x_0)^m
   \]
   converges in $B((t_0,x_0),r)$ and coincides with $F(t,x)$ on $B((t_0,x_0),r)\cap \big([0,T]\times[0,L]\big)$.
\end{definition}

It is a well-known fact that analytic functions defined on a compact domain can be extended slightly beyond its boundary:

\begin{proposition}\label{prop:analytic}
    Let $F:[0,T]\times[0,L]\to\mathbb R$ be an analytic function. Then, for $\eps$ small enough, $F$ can be extended to an analytic function on $[-\eps,T+\eps]\times[-\eps,L+\eps]$.     
\end{proposition}

\begin{proof}

Set $K:=[0,T]\times[0,L]$. By definition, for each $p=(t_0,x_0)\in K$, there exist $r_p>0$ and a power series
\[
\sum_{n,m\ge0} a_{n,m}^{(p)}(t-t_0)^n(x-x_0)^m
\]
converging on $B(p,r_p)$ and representing $F$ on $B(p,r_p)\cap K$.

Define
\[
F_p(t,x)
:=
\sum_{n,m\ge0}
a_{n,m}^{(p)}(t-t_0)^n(x-x_0)^m,
\qquad (t,x)\in B(p,r_p).
\]
Then $F_p$ is analytic on $B(p,r_p)$ and extends $F$ locally around $p$.

We claim that these local extensions agree on overlaps.
Indeed, let $p,q\in K$ and suppose that
\[
B(p,r_p)\cap B(q,r_q)\neq\varnothing .
\]
 Since $K$ is a closed rectangle with non-empty interior, the set $K\cap B(p,r_p)\cap B(q,r_q)$ contains a non-empty open neighborhood in $\mathbb{R}^2$. Both $F_p$ and $F_q$ are analytic on $B(p,r_p)\cap B(q,r_q)$, and satisfy $F_p = F_q = F$ on $K\cap B(p,r_p)\cap B(q,r_q)$.

By the identity theorem for real analytic functions, $F_p = F_q $ on the entire connected open set $B(p,r_p)\cap B(q,r_q)$.

The family $\{B(p,r_p/2)\}_{p\in K}$ is an open cover of the compact
set $K$. Hence there exist finitely many points
$p_1,\ldots,p_N\in K$ such that
\[
K\subset \bigcup_{i=1}^N B(p_i,r_i/2),
\qquad r_i:=r_{p_i}.
\]
Set
\[
U:=\bigcup_{i=1}^N B(p_i,r_i).
\]
Since the local extensions agree on overlaps, they glue together to
define an analytic function
\[
\widetilde F:U\to\mathbb R
\]
satisfying
\[
\widetilde F|_K = F.
\]

Now $U$ is an open neighborhood of the compact set $K$.
Consequently,
\[
d:=\frac{\operatorname{dist}(K,U^c)}{\sqrt{2}}>0.
\]
Choose any $\varepsilon$ with $0<\varepsilon<d$. Then every point whose
Euclidean distance from $K$ is strictly less than $\sqrt{2}\varepsilon$ belongs to $U$. In
particular,
\[
[-\varepsilon,T+\varepsilon]
\times
[-\varepsilon,L+\varepsilon]
\subset U.
\]
Restricting $\widetilde F$ to this rectangle yields the desired
analytic extension of $F$.
\end{proof}

\begin{remark}\label{rk:coerc}
By continuity, if $F\geq a_0$ on $[0,T]\times[0,L]$ for some $a_0>0$, we have that for $\eps$ small enough, its extension $\widetilde{F}$  satisfies $\widetilde{F}\geq \frac{a_0}{2}$ on $[-\eps,T+\eps]\times[-\eps,L+\eps]$.
\end{remark}

\bibliographystyle{alpha}

\bibliography{Nodalheat}

\end{document}